\documentclass[final,3p,times,number,sort&compress]{elsarticle}

\usepackage{amsmath,amssymb,amsfonts,amsthm}
\usepackage{mathtools}
\usepackage{bm}
\usepackage{algorithm}
\usepackage{algorithmic}
\usepackage{graphicx}
\graphicspath{{./}}
\usepackage{booktabs}
\usepackage{hyperref}
\usepackage{cleveref}
\usepackage{subcaption}
\usepackage{xcolor}
\usepackage{multirow}
\usepackage{siunitx}
\usepackage{enumitem}
\usepackage{lineno}
\usepackage{etoolbox}

\AtBeginEnvironment{table}{\linespread{1}\selectfont}
\AtBeginEnvironment{table*}{\linespread{1}\selectfont}
\AtBeginEnvironment{figure}{\linespread{1}\selectfont}
\AtBeginEnvironment{figure*}{\linespread{1}\selectfont}
\AtBeginEnvironment{algorithm}{\linespread{1}\selectfont}
\AtBeginEnvironment{thebibliography}{\linespread{1}\selectfont}
\modulolinenumbers[5]

\DeclareMathOperator*{\argmin}{arg\,min}
\DeclareMathOperator*{\argmax}{arg\,max}

\DeclareMathOperator{\support}{h}
\newcommand{\R}{\mathbb{R}}

\newcommand{\norm}[1]{\left\lVert #1 \right\rVert}

\newcommand{\state}{\bm{x}}
\newcommand{\ctrl}{\bm{u}}
\newcommand{\pos}{\bm{r}}

\newcommand{\reachset}{\mathcal{R}}
\newcommand{\ctrlset}{\mathcal{U}}

\newcommand{\obstset}{\mathcal{O}}

\newtheorem{theorem}{Theorem}
\newtheorem{proposition}[theorem]{Proposition}

\theoremstyle{definition}
\newtheorem{definition}[theorem]{Definition}
\newtheorem{remark}[theorem]{Remark}

\journal{Aerospace Science and Technology}

\begin{document}

\begin{frontmatter}

\title{Convex Pursuit--Evasion Games for Spacecraft Proximity Operations on Circular and Elliptical Orbits\tnoteref{draft}}
\tnotetext[draft]{Accepted for publication at \emph{Aerospace Science and Technology}, 11 September 2026.}

\author[inst1]{Omer Burak Iskender\corref{cor1}}
\ead{iske0001@e.ntu.edu.sg}

\cortext[cor1]{Corresponding author.}
\address[inst1]{School of Electrical and Electronic Engineering, Nanyang Technological University, Singapore}

\begin{abstract}
Inspecting or servicing a non-cooperative spacecraft is a two-player
game: the target can thrust to defeat the inspector's plan.
Hamilton--Jacobi--Isaacs reachability answers such games exactly but
its cost grows exponentially with the state dimension, while
learning-based controllers scale yet certify nothing.  Convex
formulations are the usual escape, and are almost always cast as
convex--concave saddle-point problems.  That framing does not survive
contact with the terminal-distance orbital
game: once a small effort regularization is included the payoff is
convex in the evader's controls as well as the pursuer's, so no pure
open-loop saddle point need exist.  In its place we derive two exact
one-sided guarantees, both read from the two players' terminal
reachable sets through their support
functions: an escape certificate that proves the target
can hold a guaranteed standoff against every admissible inspector
control, and a security strategy that bounds the miss distance
the inspector can force against every admissible target.  Together they
bracket the engagement without assuming either player is rational, and
over two hundred perturbed trials the escape certificate separated
capture from escape without a single error.  A projected extragradient
method then supplies a concrete strategy pair in about twenty-five
milliseconds, certified in place by its best-response gaps rather than
by a convergence theorem that does not apply in this regime.  The same
construction carries unchanged to elliptical reference
orbits through the Yamanaka--Ankersen dynamics, where the orbital phase
at which an engagement begins moves the miss distance by nearly a
factor of two, a dependence reproduced by nonlinear Keplerian
propagation.  Under receding-horizon play the inspector captures in
half of ten representative engagements.  Angles-only
navigation error, keep-out zones, multiple pursuers and closed-loop
play are treated as extensions, each with the
guarantee it does and does not carry.
\end{abstract}


\end{frontmatter}


\section{Introduction}\label{sec:intro}

Consider an inspector spacecraft dispatched to image a defunct
satellite at close range in low Earth orbit.  A rendezvous plan
computed against a stationary target is optimistic here, because it
never accounts for the move the target will actually make.

Rendezvous targets are conventionally graded by how much they help the
chaser~\cite{Burak2020}, from \emph{cooperative} (actively assisting the
approach) through \emph{collaborative} (built to be captured, with
markers and grapple fixtures) to \emph{uncooperative} (no assistance, no
communication link, attitude motion unknown in advance).  The engagement
studied in this paper
sits one step beyond that scale: the target is not merely unhelpful but
adversarial, since if it retains attitude control and even modest
translational authority it can thrust to keep the inspector's camera
off its damaged face.  Missions of exactly this kind (orbital
inspection, on-orbit servicing, and active debris removal) are moving
from studies to flight, and the population of maneuverable,
non-cooperative resident space objects that they must approach is
growing~\cite{Weeden2014,Hobbs2020,Long2026AutoML,Cheng2025ast_review}.
The honest way to plan against such a target is to treat inspector and
target as the two players of a differential game~\cite{Isaacs1965}, in
which each optimizes against the other's best response.

Solving that game is the hard part.  Hamilton--Jacobi--Isaacs (HJI)
reachability~\cite{Mitchell2005,Tomlin2000,Bansal2017} computes exact
equilibria but discretizes the state space, so its cost grows
exponentially with dimension and it is confined in practice to the
four-state planar problem, well short of the eight and twelve states
that two-player and multi-pursuer orbital engagements require.  Deep
reinforcement learning~\cite{Li2023cja,Yang2025cja_rl_review} scales
to those dimensions but returns a policy with no equilibrium guarantee
and no bound on how it behaves outside its training distribution, an
uncomfortable position for a safety-critical maneuver.  What is
missing is a formulation that keeps a usable guarantee, runs at the
millisecond timescale an onboard planner needs, and is not tied to a
circular reference orbit.

This paper supplies one for the terminal-distance game, in which the
inspector minimizes and the target maximizes their separation at a
fixed final time, and it does so by exploiting a structural fact.
Because the
Hill--Clohessy--Wiltshire (HCW)~\cite{Clohessy1960} and
Yamanaka--Ankersen (YA)~\cite{Yamanaka2002} relative-motion models are
linear, each player's terminal position ranges over a reachable set
that is an exact function of its control bounds, and the game is
decided by the geometry of those two sets.  That geometry is what we
compute.  It yields two guarantees, each exact and each computed from
support functions~\cite{Althoff2021,Girard2005} rather than from a
grid: a scalar \emph{escape certificate} $\phi_N$ whose negativity
proves the target can hold a guaranteed standoff against \emph{every}
admissible inspector maneuver, and a \emph{security strategy} whose
value bounds the miss distance the inspector can force against
\emph{every} admissible target.  Between the two lies an interval that
brackets the engagement outcome without assuming either player is
rational.  Across $200$ perturbed trials the escape certificate
separates capture from escape with no error
(\cref{sec:reachable,sec:solver_char}).  A projected extragradient
method~\cite{Korpelevich1976,Nemirovski2004} then returns a concrete
strategy pair in about \SI{25}{\milli\second}, roughly forty times
faster than rebuilding the iterative-best-response programs from cold
start; that ratio is a statement about cold-start cost and not about the
method being intrinsically faster than a well-implemented QP loop, as
\cref{sec:solver_char} makes precise.  We certify the quality of that
pair in place using its best-response gaps, which hold pursuer-side
optimality to under a metre (\cref{sec:game,sec:case_a}), rather than
asserting an equilibrium the problem may not possess.  Because the
construction depends on the dynamics only through the terminal support
functions, the same solver spans circular and elliptical reference
orbits, and the eccentricity dependence it predicts is confirmed
against nonlinear Keplerian propagation (\cref{sec:elliptical}).

The paper therefore contributes the following.
\begin{enumerate}[leftmargin=*,itemsep=1pt,topsep=3pt]
\item Two \emph{exact} one-sided guarantees for the terminal-distance
  game, an escape certificate and a pursuer security bound, both read
  from support functions and both valid against every admissible
  opponent (\cref{sec:capture_cert,sec:security}).
\item A curvature result showing the payoff is convex in \emph{both}
  players, so that no pure open-loop saddle is claimed, and a
  millisecond solver certified per run by its best-response gaps in
  place of a convergence theorem that does not apply here
  (\cref{sec:saddle_form,sec:extragradient}).
\item One construction that spans circular and elliptical reference
  orbits without modification, validated against nonlinear Keplerian
  and $J_2$ propagation (\cref{sec:elliptical}).
\item An account of the operational extensions, keep-out zones,
  multiple pursuers, navigation error and receding-horizon play, that
  states the guarantee each does and does not carry
  (\cref{tab:assumptions,sec:scope,sec:discussion}).
\end{enumerate}

\paragraph{Scope}
The exact guarantees above concern the open-loop game with box-bounded
controls.  Keep-out avoidance, multiple pursuers and closed-loop play are
extensions that carry weaker claims;
\cref{tab:assumptions,sec:scope} pair each assumption with the guarantee
it supports, and \cref{sec:nonlinear_validation,sec:solver_char} bound
what linearization error and relative-navigation error do to the
outcome.

\section{Related Work}\label{sec:related}

\subsection{Orbital Pursuit--Evasion}

Differential game theory for spacecraft pursuit--evasion (PE) dates to
early studies of satellite interception under Keplerian
dynamics~\cite{Ho1965,Guelman1990}.  The numerical line runs through
direct collocation for the impulsive two-player
game~\cite{Stupik2012}, canonical three-dimensional
treatments~\cite{Pontani2009}, nonlinear programming in the Hill
frame~\cite{Jagat2017}, fuel-constrained impulsive
analysis~\cite{Woodford2023}, and primer-vector methods for continuous
low thrust~\cite{Li2020}; a recurring finding is that the chaser/target
thrust-ratio asymmetry, rather than any single control law, sets the
outcome~\cite{Ye2021,Liu2020ast_thrust}.  Recent work has widened the
setting to cislunar three-body dynamics~\cite{Hafer2024} and to
reachable-set and game-theoretic maneuver strategies for space-domain
awareness~\cite{Geng2023,Cavalieri2023}.  Two contributions sit
particularly close to this paper: Ma and
Zhang~\cite{Ma2024ast_deltaV} screen impulsive engagements by the
$\Delta V$ needed to cover the target's reachable domain, a
budget-side counterpart to the support-function screen used here, and
Huo et al.~\cite{Huo2024ast_encirclement} build a multi-pursuer
encirclement strategy from the target's reachable area under maneuver
uncertainty, a geometry close to the weighted surrogate of
\cref{sec:multipursuer}.

Cutting across that literature is the older division between indirect
and direct solution methods, and it is worth saying plainly where this
paper sits.  Indirect methods apply the calculus of variations to the
game: they write the Hamiltonian for both players, derive the Euler
--Lagrange and saddle-point necessary conditions, and solve the
resulting two-point boundary-value problem in the state and costate.
This is the classical route~\cite{Isaacs1965,Ho1965} and it remains in
use for orbital engagements, through multiple-shooting and
collocation~\cite{Stupik2012,Pontani2009}, primer-vector analysis for
continuous low thrust~\cite{Li2020}, and analytically exact gradients
for the elliptical game~\cite{Pang2024PreciseGradient}.  Its appeal is
sharpness: the necessary conditions hold at the true equilibrium, with
no discretization of the control history.  Its costs are equally
well known: the boundary-value problem is sensitive to the initial
costate guess, the switching structure has to be assumed in advance,
and the solve time is minutes rather than milliseconds, so nothing is
returned that could run onboard.  The present paper is a direct method
in that taxonomy.  It discretizes the controls, keeps the resulting
problem convex on the pursuer side, and pays for the discretization
with a finite-dimensional problem that solves in tens of milliseconds
and carries per-run certificates.  The certificates are what replace
the sharpness of the necessary conditions: instead of asserting an
equilibrium, we bound how far the returned pair can be from one.  The
two families are complementary, and an indirect refinement of a direct
solution is a sensible pairing that we do not pursue here.

Two lines of recent work bear directly on the present method, and it
is useful to be specific about the trade-offs.  The first pursues
online tractability.  Jia et al.~\cite{Jia2025ClosedLoop} reduce the
elliptical-orbit game to a discounted minimum-time transfer, solve it
through a sequence of convex subproblems, and synthesize a
receding-horizon pursuit law; a companion
paper~\cite{Jia2025Reachability} fixes the terminal time from an
analytic ellipsoidal reachable-set boundary.  This is the closest
comparator to our approach and it runs in real time, but the
reduction is one-sided: it optimizes the pursuer against a
worst-case evader and returns no two-sided certificate, whereas our
construction bounds both players' guaranteed outcomes.  Pang et
al.~\cite{Pang2024PreciseGradient} solve the elliptical game to high
accuracy with analytically exact gradients, at the price of a
two-point boundary-value problem that is sensitive to its initial
guess and yields no feedback law.  The second line pursues richer
game structure: impulsive Stackelberg
equilibria~\cite{LiLuo2025Stackelberg}, three-player and two-on-one
engagements~\cite{Sun2025ThreePlayer,Zhao2025TwoOnOne}, and analytic
winning regions for target--attacker--defender
games~\cite{Fu2026TAD,Yang2026WinningRegions}.  These enlarge what can
be modeled but rely on closed-form or search-based solvers without the
polynomial-time guarantees of the convex problem solved here.
Learning-based controllers are a third, orthogonal line, surveyed
by Yang et al.~\cite{Yang2025cja_rl_review} and exemplified by Li et
al.~\cite{Li2023cja}, with recent multi-agent, distributed, and
incomplete-information
variants~\cite{Hu2024ast_multiagent,Huang2026ast_diverse,LiYe2022ast_switching,Wang2025LSTM};
they scale to high dimension but certify nothing.  Cheng et
al.~\cite{Cheng2025ast_review} review the broader field.  Against this
backdrop the present work occupies a specific niche: exact one-sided
guarantees and millisecond runtime on linear relative dynamics, for
both circular and elliptical orbits.  \Cref{tab:comparison} summarizes
the comparison.

\begin{table}[htbp]
\centering\small
\caption{Methodological position relative to representative spacecraft
  PE methods.}
\label{tab:comparison}
\resizebox{\textwidth}{!}{%
\begin{tabular}{@{}llllll@{}}
\toprule
 & This work & Jia~\cite{Jia2025ClosedLoop} & Stupik~\cite{Stupik2012} & Li~\cite{Li2023cja} & Hafer~\cite{Hafer2024} \\
\midrule
Method        & Reachable-set + EG & Convex min-time & Direct collocation & DRL & HJI \\
Guarantees    & Two one-sided bounds & One-sided & Local optimum & None & Saddle-point \\
Runtime       & $\sim$25~ms     & real-time     & minutes       & seconds & minutes \\
Orbit type    & Circ.+Ellip.    & Elliptical    & Circular      & Circular & CR3BP \\
Dimension     & 4D planar, 6D verified & 6D     & 6D            & 6D      & 6D \\
Players       & $n_P$ vs.\ 1    & 1 vs.\ 1      & 1 vs.\ 1      & 1 vs.\ 1 & 1 vs.\ 1 \\
Reachable set & Exact (support fn.) & Ellipsoidal & No          & No      & No \\
\bottomrule
\end{tabular}%
}
\end{table}

\subsection{Relative Motion on Elliptical Orbits}\label{sec:related_elliptical}

Linearized relative motion about an elliptical reference orbit goes
back to Tschauner and Hempel~\cite{Tschauner1965}; the state
transition matrices that followed~\cite{Carter1998,Broucke2003}
carried integrals that grow singular near circular orbits, which
Yamanaka and Ankersen~\cite{Yamanaka2002} removed with a
$J$-integral formulation well conditioned at every eccentricity.  That
line has since been extended to eccentric formation
control~\cite{Inalhan2002}, perturbed non-circular
STMs~\cite{Gim2003}, and orbit-element formulations of nonlinear
Riccati formation dynamics~\cite{Vazquez2021ast_riccati}, and is
surveyed by Sullivan et al.~\cite{Sullivan2017}.  Two recent results
bear directly on \cref{sec:elliptical}: an analytical
linear-quadratic PE strategy for arbitrary Keplerian reference orbits
via the Tschauner--Hempel differential Riccati
equation~\cite{Fu2025ast_keplerian}, and a theory-of-functional-connections
treatment of analytic rendezvous boundary conditions in elliptical PE
problems~\cite{Zhang2025ast_tfc}.  We adopt the Yamanaka--Ankersen STM
because its conditioning across the whole eccentricity range is what
lets one solver cover circular and elliptical orbits alike.

\subsection{Reachability for Linear Orbital Systems}

Hamilton--Jacobi (HJ) reachability~\cite{Mitchell2005,Tomlin2000}
provides rigorous backward reachable tubes that characterize the set
of states from which one player can guarantee a particular outcome;
the approach has been applied to aerial pursuit--evasion and
multi-player reach--avoid
problems~\cite{Chen2018,Fisac2015,Bansal2017,Herbert2017}.  Their grid
cost is the exponential one already noted in \cref{sec:intro}.

For linear systems with bounded inputs, forward reachable sets admit
exact set-valued representations that avoid this curse, among them
zonotopes~\cite{Girard2005}, ellipsoids~\cite{Kurzhanski2000}, and
support-function polytopes~\cite{Althoff2021}; support functions are
attractive here because they separate the set representation from the
propagation, reducing the reachable set to scalar evaluations along
template directions.  In the spacecraft setting the domain has been
approximated polyhedrally for relative
motion~\cite{Shao2023cja_reachable} and extended to multi-impulse
elliptical maneuvers~\cite{Zhang2025cja_reachable}.  Closest to the
present use, several groups turn the reachable domain into a
capture-feasibility test, through the $\Delta V$ needed to cover the
target's set~\cite{Ma2024ast_deltaV}, through time-dependent reachable
domains~\cite{Zhang2025TimeDependentRD}, or by feeding a
reachable-domain screen into equilibrium
computation~\cite{Xu2026NashReachable,Jia2025Reachability}.  A related
use is to delimit safe \emph{start} regions rather than capture sets,
as in approach to a tumbling target under rotating line-of-sight
constraints~\cite{iskender2026reachability}.  All use
ellipsoidal or sampled approximations; the escape certificate of
\cref{sec:capture_cert} instead uses the \emph{exact} support function
of the linear reachable set, which is what lets it certify rather than
estimate infeasibility.

\subsection{Saddle-Point Methods and Game Solvers}

Existence and uniqueness of equilibria in continuous games were
established by Rosen~\cite{Rosen1965} for concave $N$-person games and
by the minimax theorem of von Neumann and
Morgenstern~\cite{vonNeumannMorgenstern1944} for the two-player
zero-sum convex--concave case.  Neither applies here: the
terminal-distance payoff falls outside the convex--concave class
(\cref{prop:curvature}), which is why the treatment below is
certificate-based rather than an appeal to a minimax theorem.
Iterative best response (IBR), in which the players optimize
against each other in alternation, is a common practical heuristic and
reduces each subproblem to a standard convex program for off-the-shelf
QP/SOCP solvers~\cite{BoydVandenberghe2004,OSQP2020}, but it lacks
convergence guarantees and can cycle.

The extragradient method of Korpelevich~\cite{Korpelevich1976} and its
variants~\cite{Nemirovski2004,Tseng1995,Mokhtari2020} solve monotone
variational inequalities, of which convex--concave saddle problems are
a special case, and its two-step extrapolation removes the oscillation
of naive gradient descent--ascent; Facchinei and
Pang~\cite{Facchinei2003} give the standard reference.  We retain it as
a fast equilibrium-seeking solver, since projection onto box control
sets is element-wise clipping (\cref{sec:extragradient}), while being
explicit that its convergence theory does not apply in our operating
regime and using the certificates to judge the solution instead.  For constrained
games the alternatives are augmented-Lagrangian solvers such as
ALGAMES~\cite{LeCleach2022ALGAMES} and epigraph
reformulations~\cite{SoFan2023Epigraph}, which handle constraints
exactly but trade away global convergence; keep-out constraints here
are instead treated by successive convexification (\cref{sec:obstacle}),
a technique with a strong track record in aerospace trajectory
optimization~\cite{Malyuta2022}.

Convex programming underpins much of aerospace trajectory
generation~\cite{AcikmeseaPloen2007}, and model predictive control
extends it to single-agent rendezvous and
docking~\cite{Rawlings2017,Burak2020,Iskender2019}, including robust
tube-based~\cite{Mammarella2018ast_tubeMPC} and chance-constrained
variants for halo-orbit rendezvous~\cite{Sanchez2020ast_NRHO} and
convex deterministic tightenings of probabilistic
constraints~\cite{Oguri2024ChanceConstrained}.  All assume a passive or
worst-case target.  The present paper keeps the convex structure that
makes such controllers fast but replaces that assumption with an
explicit game; a nominal-MPC baseline
(\cref{sec:experiments}) measures the cost of wrongly assuming
passivity.

\section{Dynamics and Problem Formulation}\label{sec:dynamics}

\subsection{Relative Motion Model}

Relative motion between a target (chief) on a circular reference orbit
and a deputy (chaser) is modeled by the Hill--Clohessy--Wiltshire (HCW)
equations~\cite{Clohessy1960,Hill1878} in the local-vertical
local-horizontal (LVLH) frame ($x$ radial, $y$ along-track, $z$
cross-track):
\begin{equation}\label{eq:hcw_3d}
\begin{aligned}
\ddot{x} - 2n\dot{y} - 3n^2 x &= u_x, \\
\ddot{y} + 2n\dot{x} &= u_y, \\
\ddot{z} + n^2 z &= u_z,
\end{aligned}
\end{equation}
with mean motion $n = \sqrt{\mu/a^3}$, where $\mu$ is the Earth's
gravitational parameter and $a$ the reference semi-major axis.  In
state-space form $\dot{\state} = A_c \state + B_c \ctrl$.  The
cross-track axis decouples, so all experiments use the planar model
with $n_x = 4$ states $\state = [x, y, \dot{x}, \dot{y}]^\top$ and
$n_u = 2$ controls $\ctrl = [u_x, u_y]^\top$; the $n_p = 2$ position
components are $\pos = [x, y]^\top$.  A zero-order hold discretization
at timestep $\Delta t$ yields
\begin{equation}\label{eq:discrete}
\state_{k+1} = A_d\, \state_k + B_d\, \ctrl_k, \quad k = 0, \ldots, N-1,
\end{equation}
with $(A_d, B_d)$ obtained from the matrix exponential of the
augmented block $\bigl[\begin{smallmatrix} A_c & B_c \\ 0 & 0 \end{smallmatrix}\bigr]\Delta t$.
The transition matrix is standard~\cite{Clohessy1960}, and
\cref{app:params} records the discretization check: against an RK45
reference the ZOH model is within $0.01\%$ at the $\Delta t = 10$~s used
throughout.

Throughout, $\ctrl_k$ is an \emph{acceleration} in
\si{\meter\per\second\squared}, held constant over the interval
$[k\Delta t, (k+1)\Delta t)$ by the zero-order hold, and never an
impulsive velocity increment.  The distinction matters for reading the
tables: the thrust bounds $\bar{u}$ are quoted in
\si{\milli\meter\per\second\squared}, whereas the reported
$\Delta v = \sum_k \norm{\ctrl_k}_2 \Delta t$ is the integrated
velocity change in \si{\meter\per\second}, which is why a bound of
$10$~\si{\milli\meter\per\second\squared} over $300$~\si{\second}
appears as a few \si{\meter\per\second} of $\Delta v$.  An impulsive
formulation would keep everything that follows intact, with $B_d$
replaced by the velocity-injection map and $\Delta t$ dropped from the
$\Delta v$ sum; continuous thrust is used because the box bound is then a
bound on the acceleration a body-mounted thruster set delivers directly.

The cross-track reduction is verified rather than assumed:
\cref{sec:scope_checks} solves the same engagements on the full
six-state model with out-of-plane offsets and recovers the planar
numbers exactly when the cross-track state is zero.

\subsection{Linearization Validity}\label{sec:nonlinear_validation}

Whether the linear model can be trusted depends on the engagement
regime, so we quantify its error over one rather than assert a single
bound.  A $360$-case sweep compares the discrete HCW model against
full nonlinear Earth-centered-inertial propagation (two-body and
two-body-plus-$J_2$) across altitudes of $400$--$800$~km, initial
separations of $0.1$--$2$~km, closing speeds up to $1$~\si{\meter\per\second},
horizons of $300$ and $600$~\si{\second}, and both free drift and
maximum thrust.  \Cref{fig:linearization_error} summarizes the result.
The error is dominated by unmodeled differential $J_2$ and grows
roughly linearly with separation and about fourfold from $300$ to
$600$~\si{\second}.  Over a $300$~\si{\second} horizon it stays below
$0.16$~\si{\meter} for separations up to $600$~\si{\meter} and below
$0.43$~\si{\meter} out to $2$~\si{\kilo\meter}; over $600$~\si{\second}
it reaches $1.76$~\si{\meter} at $2$~\si{\kilo\meter}.  Relative to the
separation this is under $0.06\%$ at $300$~\si{\second} everywhere
tested, and under $0.26\%$ at $600$~\si{\second} for separations of
$300$~\si{\meter} or more; relative to the $50$~\si{\meter} capture
radius that decides outcomes, the model error is at most $0.9\%$ over
$300$~\si{\second} and $3.5\%$ over $600$~\si{\second}, so
capture and escape are distinguished unambiguously whenever the
terminal distance is more than a few metres from the boundary.
Replaying a single Case~A engagement against the same nonlinear
propagator with $J_2$ and atmospheric drag both active agrees to the
same order in the time domain: the peak inspector position error is
$0.056$~\si{\meter}, the linear and nonlinear inter-agent distances
differ by $0.024$~\si{\meter} ($0.03\%$) at the terminal time, and
drag alone contributes under $0.2$~\si{\milli\meter}.  The one effect
the planar model cannot represent is
differential-$J_2$ cross-track motion, which reaches $0.9$~\si{\meter}
over $600$~\si{\second} at the most demanding grid point.

\begin{figure}[!htbp]
\centering
\includegraphics[width=0.894\textwidth]{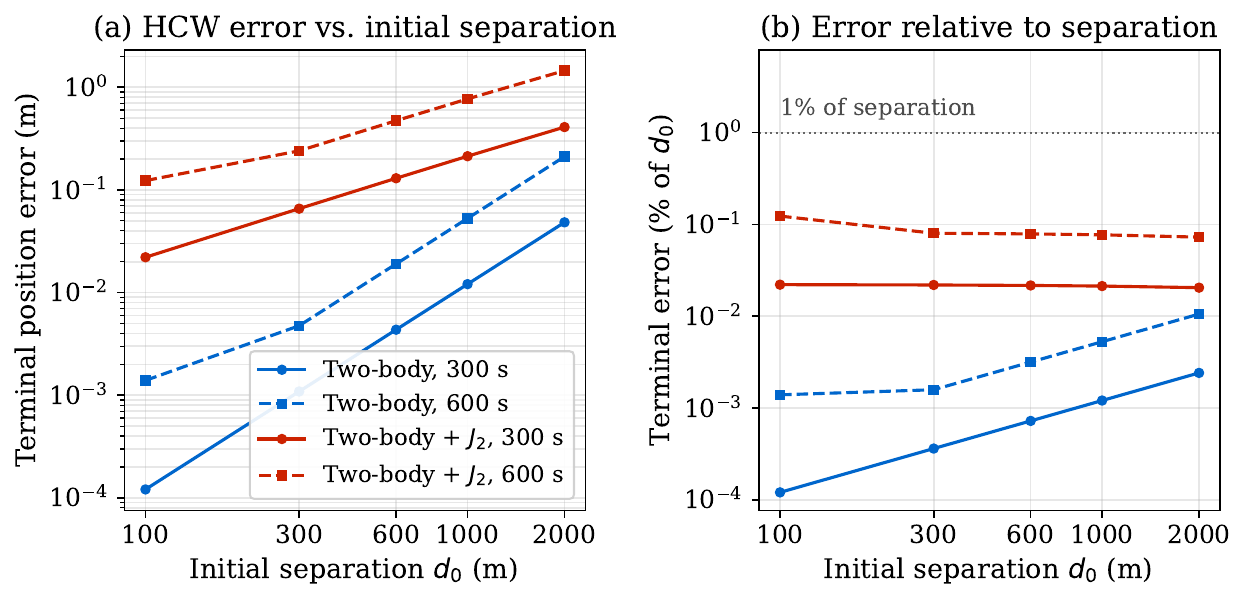}
\caption{Linearization error of the discrete HCW model against
  nonlinear ECI propagation, free-drift envelope over $400$--$800$~km
  altitude and $0$--$1$~\si{\meter\per\second} closing speed.
  (a)~Terminal position error versus initial separation for two-body
  and two-body-plus-$J_2$ truth at $300$ and $600$~\si{\second}.
  (b)~The same error as a fraction of separation; the $1\%$ reference
  line is never approached at $300$~\si{\second}.}
\label{fig:linearization_error}
\end{figure}

\subsection{Players, Controls, and the Game}\label{sec:problem}

Consider an engagement between $n_P$ pursuers
($i \in \{1,\ldots,n_P\}$) and a single evader.  Each agent
$j \in \{P_1, \ldots, P_{n_P}, E\}$ evolves under the discrete
dynamics
\begin{equation}\label{eq:dynamics_agent}
\state_k^{(j)} = A_d\, \state_{k-1}^{(j)} + B_d\, \ctrl_{k-1}^{(j)},
\quad k = 1, \ldots, N,
\end{equation}
with box-bounded controls
\begin{equation}\label{eq:ctrl_bounds}
\ctrl_k^{(j)} \in \ctrlset^{(j)} \triangleq
\bigl\{ \ctrl \in \R^{n_u} : \norm{\ctrl}_\infty \leq \bar{u}^{(j)} \bigr\},
\end{equation}
where $\bar{u}^{(P_i)} = \bar{u}_P$ for pursuers and
$\bar{u}^{(E)} = \bar{u}_E$ for the evader.  The thrust ratio
$\bar{u}_P/\bar{u}_E > 1$ throughout, reflecting the operational case
in which the inspector has a maneuvering advantage over the target.

The paper studies the terminal-distance pursuit--evasion game: the
pursuer team minimizes, and the evader maximizes, the squared distance
between the evader and the nearest pursuer at the final time,
\begin{equation}\label{eq:G1}
\min_{\{\ctrl^{(P_i)}\}} \max_{\ctrl^{(E)}} \;
\min_{i \in \{1,\ldots,n_P\}} \norm{\pos_N^{(P_i)} - \pos_N^{(E)}}_2^2,
\end{equation}
where $\pos_k^{(j)}$ is the position part of $\state_k^{(j)}$.  We
write $d_f = \min_i \norm{\pos_N^{(P_i)} - \pos_N^{(E)}}_2$ for the
resulting terminal miss distance.  Two capture conventions appear in
the paper and it helps to fix them here.  For the open-loop game,
capture at horizon $N$ means $d_f \leq r_{\mathrm{cap}}$ for a
specified capture radius.  For the closed-loop, receding-horizon play
of \cref{sec:tradespace}, capture is instead a first-passage event:
the engagement terminates at the first step whose separation enters
the capture ball, $t_{\mathrm{cap}} = \Delta t \cdot \min\{k :
\min_i\norm{\pos_k^{(P_i)}-\pos_k^{(E)}}_2 \leq r_{\mathrm{cap}}\}$,
after which the outcome is decided and later drift is irrelevant.
This is the standard game-of-kind convention~\cite{Isaacs1965} and it
is what the receding-horizon figures and tables report.
Minimum-time intercept and fuel-bounded variants lie outside the
present scope, as they introduce integer or nonlinear structure that
breaks the reachable-set geometry exploited here.

Throughout, ``inspector'' and ``pursuer'' denote the same agent $P$
and ``target'' and ``evader'' the same agent $E$.  The information
pattern is open-loop, with each player committing to a full control
sequence at $k=0$.  \Cref{tab:notation} in \cref{app:notation}
collects every symbol used in the paper.

\subsection{Scope, Assumptions, and Guarantees}\label{sec:scope}

Game-based guidance tends to see its guarantees weaken as the model is
made realistic, so \cref{tab:assumptions} states what each part of the
framework does and does not guarantee.  The strongest claims, the
escape and security certificates, rest only on the top rows, which are
exactly the assumptions the linearization can support.

\begin{table}[htbp]
\centering\small
\renewcommand{\arraystretch}{1.25}
\caption{Assumptions and the guarantee each one supports.  Section
  numbers indicate where each item is established or used.}
\label{tab:assumptions}
\begin{tabular}{@{}p{0.34\textwidth}p{0.57\textwidth}@{}}
\toprule
Assumption & Status / scope of guarantee \\
\midrule
Planar HCW motion at LEO over $300$--$600$~s horizons & Linearization error below $0.9\%$ of the capture radius against nonlinear ECI$+J_2$ over the tested regime (\cref{sec:nonlinear_validation}); the planar reduction and horizons to $1200$~s are checked on the six-state model in \cref{sec:scope_checks} \\
Linear time-varying Yamanaka--Ankersen dynamics, $e \in [0, 0.6]$ & Reachable-set and certificate constructions preserved verbatim; validated against nonlinear Keplerian$+J_2$ propagation (\cref{sec:ltv_game,sec:elliptical}) \\
Continuous thrust, $\ell_\infty$-bounded (box), no illumination, communication or sensing constraints & Makes the reachable-set support function closed-form and projection element-wise (\cref{sec:support,sec:extragradient}); which operational constraints preserve that structure is set out in \cref{sec:limitations} \\
Terminal-distance payoff, box controls & \emph{Exact} escape certificate and pursuer security bound from support functions (\cref{thm:escape,def:security}); these hold for any admissible opponent \\
Effort regularization $\lambda > 0$ & Strict pursuer-side convexity; the payoff is \emph{not} concave in the evader, so no pure open-loop saddle is claimed (\cref{prop:curvature}) \\
Projected extragradient solution & Fast equilibrium-seeking; solution quality certified per run by best-response gaps, not by a convergence theorem in this regime (\cref{sec:extragradient,sec:case_a}) \\
Multi-pursuer via weighted surrogate & Heuristic coordination; the rigorous statement is the joint escape certificate (\cref{sec:multipursuer,sec:case_c}) \\
Keep-out avoidance via successive convexification & Exact constraint satisfaction at convergence; each convex subproblem is well posed (\cref{sec:obstacle,sec:case_b}) \\
Receding-horizon closed-loop play & Empirical; no recursive-feasibility proof, but per-step feasibility is checkable from the reachable-set screen (\cref{sec:tradespace}) \\
Perfect state knowledge, deterministic dynamics & Modeling assumption; navigation-error sensitivity measured in \cref{sec:solver_char}, further relaxations in \cref{sec:discussion} \\
\bottomrule
\end{tabular}
\end{table}

\section{Reachable Set Interpretation}\label{sec:reachable}

\subsection{Support-Function Propagation of Forward Reachable Sets}\label{sec:support}

The $N$-step forward reachable set under \cref{eq:discrete} with
control set $\ctrlset$ admits the Minkowski decomposition
\begin{equation}\label{eq:reach_mink}
\reachset_N(\state_0) = A_d^N \state_0 \oplus
\bigoplus_{k=0}^{N-1} A_d^{N-1-k} B_d \ctrlset.
\end{equation}
Defining the support function
$\support_{\mathcal{S}}(\bm{d}) \triangleq \max_{\state \in \mathcal{S}} \bm{d}^\top \state$
and using its standard linear-map and Minkowski
properties~\cite{Althoff2021},
\begin{equation}\label{eq:sf_reach}
\support_{\reachset_N}(\bm{d}) =
\bm{d}^\top A_d^N \state_0
+ \sum_{k=0}^{N-1} \support_{\ctrlset}\!\left( B_d^\top (A_d^\top)^{N-1-k} \bm{d} \right).
\end{equation}
For the box control set $\ctrlset = \{\ctrl : \norm{\ctrl}_\infty \leq \bar{u}\}$
the support function is $\bar{u} \norm{\bm{d}}_1$, which evaluates in
$O(n_u)$ time.  The position-space reachable set is then
outer-approximated by evaluating \cref{eq:sf_reach} along $L$ template
directions $\{\bm{d}_\ell\}$ uniformly distributed on the unit circle:
\begin{equation}\label{eq:polytope}
\hat{\reachset}_N \triangleq
\bigcap_{\ell=1}^{L} \bigl\{\pos \in \R^{n_p} :
\bm{d}_\ell^\top \pos \leq \support_{\reachset_N}(\bm{d}_\ell)\bigr\}
\supseteq \reachset_N^{\mathrm{pos}}.
\end{equation}
Sweeping $L \in \{8, \ldots, 128\}$, the excess area ratio
$\alpha_L = (V_L - V^*)/V^*$ relative to a fine reference ($L = 1000$)
decreases as $O(1/L^2)$; at $L = 48$, $\alpha_L \approx 0.01$.  Two
templates appear in what follows: $L = 48$ where a set is drawn or
screened, and $L = 96$ for the certificates of
\cref{sec:capture_cert,sec:security}, where the extra directions cost
microseconds and tighten the bound.

\subsection{An Exact Escape Certificate}\label{sec:capture_cert}

The support functions of \cref{eq:sf_reach} are exact for box
controls, and this is enough to certify, not merely estimate, that a
target cannot be caught.  Define, over any finite set of unit
directions $\{\bm{d}_\ell\}_{\ell=1}^{L}$,
\begin{equation}\label{eq:feas_cond}
\phi_N^{(i)} \triangleq \min_{\ell \in \{1,\ldots,L\}} \left[
\support_{\reachset_N^{(P_i)}}(\bm{d}_\ell)
- \support_{\reachset_N^{(E)}}(\bm{d}_\ell)
+ r_{\mathrm{cap}} \right],
\end{equation}
where $\support_{\reachset_N^{(j)}}$ is the exact terminal-position
support function of agent $j$, and let $\bm{d}^*$ attain the minimum.

\begin{theorem}[Escape certificate]\label{thm:escape}
If $\phi_N^{(i)} < 0$, the evader has an admissible open-loop control
that keeps
$\norm{\pos_N^{(P_i)} - \pos_N^{(E)}}_2 \geq r_{\mathrm{cap}} - \phi_N^{(i)}
> r_{\mathrm{cap}}$
against \emph{every} admissible control of pursuer $i$.  Escape from
pursuer $i$ is therefore guaranteed, with margin $-\phi_N^{(i)}$.
\end{theorem}

\begin{proof}
Let the evader play the bang-bang control that attains the support
value in direction $\bm{d}^*$, so that
$\bm{d}^{*\top}\pos_N^{(E)} = \support_{\reachset_N^{(E)}}(\bm{d}^*)$;
for box controls this is
$\ctrl_k^{(E)} = \bar{u}_E\,\mathrm{sign}(B_d^\top (A_d^\top)^{N-1-k}\bm{d}^*)$.
Any admissible pursuer terminal position obeys
$\bm{d}^{*\top}\pos_N^{(P_i)} \leq \support_{\reachset_N^{(P_i)}}(\bm{d}^*)$.
By Cauchy--Schwarz,
$\norm{\pos_N^{(E)}-\pos_N^{(P_i)}}_2 \geq
\bm{d}^{*\top}(\pos_N^{(E)}-\pos_N^{(P_i)}) \geq
\support_{\reachset_N^{(E)}}(\bm{d}^*)
-\support_{\reachset_N^{(P_i)}}(\bm{d}^*)
= r_{\mathrm{cap}} - \phi_N^{(i)}$.
\end{proof}

\begin{remark}[Existence of capture versus success of capture]\label{rem:existence}
Ask instead whether some admissible pair of plans ends close together.
That is a joint question about the two reachable sets: writing
$\mathcal{D}_N \triangleq \reachset_N^{(P)} \oplus
(-\reachset_N^{(E)})$ for their \emph{difference body}, the Minkowski
sum of $\reachset_N^{(P)}$ and $-\reachset_N^{(E)}$ and not the
erosion $\ominus$, a pair of
admissible controls with terminal separation at most
$r_{\mathrm{cap}}$ exists if and only if $\mathcal{D}_N$ meets the
capture ball of radius $r_{\mathrm{cap}}$, which for convex compact
sets is decided by
$\psi_N \triangleq \min_{\norm{\bm{d}}_2 = 1}
[\support_{\reachset_N^{(P)}}(\bm{d}) +
\support_{\reachset_N^{(E)}}(-\bm{d})] + r_{\mathrm{cap}} \geq 0$.
Note the plus sign: $\psi_N$ is not $\phi_N$, and since each terminal
set is a free response plus a zonotope centred on it,
$\phi_N \leq \psi_N$ always.  Non-existence of a capture pair therefore
implies certified escape, but not conversely, and it is the converse
that matters operationally.  In Case~A of \cref{sec:case_a} the two
terminal sets in fact \emph{intersect}: the minimum separation over all
admissible pairs, obtained exactly from a small convex program, is
$0$~\si{\meter}, so capture pairs exist in abundance.  The target
nevertheless guarantees $72.3$~\si{\meter} of standoff, because the
inspector must commit to its plan without knowing the target's.
A Pontryagin difference, which encodes the stronger requirement that one
set contain another, is likewise not the object this question needs.
\end{remark}

Two points make $\phi_N$ a genuine certificate rather than a screen.
First, it delivers not only the sign of the outcome but a
\emph{constructive} evader strategy and a guaranteed standoff.  Second,
direction sampling can only make it conservative: the minimum over a
finite $\{\bm{d}_\ell\}$ is at least the minimum over the whole circle,
so a sampled $\phi_N < 0$ is valid a fortiori, and refining $L$ can
only tighten the margin.  Against multiple pursuers, escape from all of
them simultaneously requires a single direction that clears every
pursuer at once,
\begin{equation}\label{eq:phi_joint}
\phi_N^{\mathrm{joint}} \triangleq \min_{\ell}\;\max_i\left[
\support_{\reachset_N^{(P_i)}}(\bm{d}_\ell)
- \support_{\reachset_N^{(E)}}(\bm{d}_\ell)\right] + r_{\mathrm{cap}},
\end{equation}
so $\phi_N^{\mathrm{joint}} < 0$ certifies joint escape while each
$\phi_N^{(i)} < 0$ certifies escape only from pursuer $i$ in isolation.
The gap between the two is exactly how encirclement helps the pursuers,
made quantitative in \cref{sec:case_c}.  When $\phi_N \geq 0$ the
certificate is silent and the full game must be solved;
\cref{fig:feasibility} evaluates $\phi_N$ on Case~A across thrust
ratios, and \cref{sec:solver_char} reports that over $200$ perturbed
trials the sign of $\phi_N$ predicts the game outcome without a single
error.

\section{Game Formulation and Solver}\label{sec:game}

\subsection{The Terminal-Distance Game and Its Curvature}\label{sec:saddle_form}

The linearity of \cref{eq:dynamics_agent} makes each agent's state
trajectory an affine function of its control sequence.  Stacking
controls as $\bm{U}^{(j)} \in \R^{Nn_u}$ and using the
block-lower-triangular control-to-state matrix $S$ with blocks
$S_{k,j} = A_d^{k-j-1} B_d$ for $j<k$,
\begin{equation}\label{eq:traj_affine}
\bm{X}^{(j)} = \bm{f}^{(j)} + S \bm{U}^{(j)},
\end{equation}
where $\bm{f}^{(j)}$ is the free response.  Writing $T_N$ for the
terminal-position extraction matrix, so that $T_N\bm{X}^{(j)} =
\pos_N^{(j)}$, the terminal relative position is
\begin{equation}\label{eq:delta}
\bm{\delta} = G_P \bm{U}^{(P)} - G_E \bm{U}^{(E)} + \bm{g}_0,
\end{equation}
with $G_P = G_E = T_N S$ and $\bm{g}_0 = T_N(\bm{f}^{(P)} - \bm{f}^{(E)})$.

The game payoff is the squared terminal distance with a small symmetric
effort regularization,
\begin{equation}\label{eq:payoff}
J(\bm{U}^{(P)}, \bm{U}^{(E)}) =
\norm{\bm{\delta}}_2^2
+ \lambda \norm{\bm{U}^{(P)}}_2^2
- \lambda \norm{\bm{U}^{(E)}}_2^2,
\end{equation}
with $\lambda = 10^{-3}$ throughout, and the open-loop problem is
\begin{equation}\label{eq:saddle}
\min_{\bm{U}^{(P)} \in \ctrlset_P^N} \max_{\bm{U}^{(E)} \in \ctrlset_E^N}
J(\bm{U}^{(P)}, \bm{U}^{(E)}).
\end{equation}
It is tempting to call \cref{eq:saddle} a convex--concave game and
invoke a minimax theorem, but that is only half true, and the
distinction matters for the rest of the paper.

\begin{proposition}[Curvature of the payoff]\label{prop:curvature}
For every $\lambda > 0$ the payoff $J$ is strictly convex in
$\bm{U}^{(P)}$.  It is concave in $\bm{U}^{(E)}$ if and only if
$\lambda \geq \sigma_{\max}(G_E)^2$, where $\sigma_{\max}(G_E)$ is the
largest singular value of $G_E$.  The same threshold governs
monotonicity of the game operator $F$ below.
\end{proposition}

\begin{proof}
The Hessians are
$\nabla^2_{\bm{U}^{(P)}} J = 2(G_P^\top G_P + \lambda I) \succ 0$ and
$\nabla^2_{\bm{U}^{(E)}} J = 2(G_E^\top G_E - \lambda I)$, and the
latter is negative semidefinite exactly when $\lambda I \succeq
G_E^\top G_E$, i.e.\ $\lambda \geq \sigma_{\max}(G_E)^2$.  The
bilinear cross terms of $F$ cancel in its symmetric part, which is
$\mathrm{blkdiag}(2(G_P^\top G_P + \lambda I),\,2(\lambda I - G_E^\top
G_E))$; monotonicity is equivalent to this being positive
semidefinite, the same condition.
\end{proof}

For the engagement parameters used here $\sigma_{\max}(G_E)^2 \approx
9.4\times 10^{7}$ in the units of \cref{eq:payoff}, eleven orders of
magnitude above $\lambda = 10^{-3}$.  Any $\lambda$ large enough to
restore concavity would also swamp the terminal-distance objective
with the effort term and change the game being played.  In the natural
parameterization, then, the terminal-distance game is convex in
\emph{both} players: the evader maximizes a convex function of its
terminal position over its reachable set, so its optimal reply is a
bang-bang extreme point and a pure open-loop saddle need not exist.
The rest of the paper is built around that fact.

The value $\lambda = 10^{-3}$ is therefore chosen as a regularizer, not
as a tuning knob, and the outcome is insensitive to it.  Sweeping
$\lambda$ over twelve values from $10^{-4}$ to $10^{7}$ leaves Case~A
unchanged at $d_f = 72.63$~\si{\meter} and
$\Delta V_P = 4.16$~\si{\meter\per\second} across the whole range
$[10^{-4}, 10^{4}]$, because both players saturate their thrust bounds
($96.7\%$ of the horizon) and the effort term merely shifts the
objective by a constant.  The first departure exceeding $1\%$ appears
only near $\lambda^\star \approx 7 \times 10^{5}$, nearly nine orders
above the
operative value: by $\lambda = 10^{7}$ the penalty desaturates the
thrust and $d_f$ rises to $109.6$~\si{\meter}.  A positive $\lambda$ is
kept because it makes the pursuer subproblem strongly convex and the
projection well posed, so the best response is unique; the full sweep
is in \cref{app:ablation}.

The gradients
\begin{align}
\nabla_{\bm{U}^{(P)}} J &= 2 G_P^\top \bm{\delta} + 2\lambda \bm{U}^{(P)},
  \label{eq:grad_p} \\
\nabla_{\bm{U}^{(E)}} J &= -2 G_E^\top \bm{\delta} - 2\lambda \bm{U}^{(E)},
  \label{eq:grad_e}
\end{align}
define the game operator
$F = [\nabla_{\bm{U}^{(P)}} J,\; -\nabla_{\bm{U}^{(E)}} J]^\top$, which
is Lipschitz continuous with constant
$L_F = 2(\norm{G}_2^2 + \lambda)$, $\norm{G}_2 = \max\{\norm{G_P}_2,
\norm{G_E}_2\}$.

\subsection{Two Exact One-Sided Guarantees}\label{sec:security}

In place of an equilibrium, each player is furnished with a worst-case
guarantee that holds regardless of the opponent.  The evader's is the escape
certificate of \cref{thm:escape}: when $\phi_N < 0$ it can force a
terminal separation of at least $r_{\mathrm{cap}} - \phi_N$.  The
pursuer's is a security strategy.

\begin{definition}[Pursuer security strategy]\label{def:security}
Let $\hat{\reachset}_N^{(E)} \supseteq \reachset_N^{(E)}$ be the outer
polytope \cref{eq:polytope} of the evader's terminal reachable set,
with vertices $\{\bm{v}_j\}_{j=1}^{L}$.  The pursuer security value is
\begin{equation}\label{eq:security}
\bar{V} = \min_{\bm{U}^{(P)} \in \ctrlset_P^N}\; \max_{j}\;
\norm{\pos_N^{(P)}(\bm{U}^{(P)}) - \bm{v}_j}_2,
\end{equation}
attained by the security control $\bm{U}^{(P)}_{\mathrm{sec}}$.
\end{definition}

Since the polytope contains the true reachable set, playing
$\bm{U}^{(P)}_{\mathrm{sec}}$ guarantees a terminal miss distance of at
most $\bar{V}$ against every admissible evader, and $\bar{V} \leq
r_{\mathrm{cap}}$ certifies capture.  Problem \cref{eq:security} is a
second-order-cone program with $L$ cones, a one-time solve rather than
a game iteration.  The two guarantees bracket the engagement: when
$\phi_N < 0$, the outcome lies in $[\,r_{\mathrm{cap}} - \phi_N,\,
\bar{V}\,]$, a certified interval that assumes nothing about either
player's rationality or information.

\subsection{Extragradient Equilibrium-Seeking, Certified per Run}\label{sec:extragradient}

The certificates bound the outcome but do not by themselves produce a
strategy to fly.  For that we run the projected extragradient
method~\cite{Korpelevich1976} on \cref{eq:saddle}.  With
$\bm{z}_t = (\bm{U}_t^{(P)}, \bm{U}_t^{(E)})$, each iteration takes an
extrapolation step
\begin{equation}\label{eq:eg_extrap}
\hat{\bm{z}}_t = \Pi_{\mathcal{Z}}\bigl(\bm{z}_t - \eta\, F(\bm{z}_t)\bigr)
\end{equation}
and an update step
\begin{equation}\label{eq:eg_update}
\bm{z}_{t+1} = \Pi_{\mathcal{Z}}\bigl(\bm{z}_t - \eta\, F(\hat{\bm{z}}_t)\bigr),
\end{equation}
with $\mathcal{Z} = \ctrlset_P^N \times \ctrlset_E^N$ and step size
$\eta = \min(\eta_0, 0.5/(\norm{G}_2^2 + \lambda))$.  Because the
constraints are boxes, $\Pi_{\mathcal{Z}}$ is element-wise clipping at
$O(Nn_u)$ cost, and the evader block of $F$ carries a sign flip so its
update is a gradient ascent.  Each iteration is two gradient
evaluations and two projections, all matrix--vector work; no QP solver
is called, which is what makes it fast (\cref{sec:experiments}).

\begin{theorem}[Conditional convergence]\label{thm:convergence}
If $\lambda \geq \sigma_{\max}(G_E)^2$, so that $F$ is monotone
(\cref{prop:curvature}), then for $\eta \leq 1/(2L_F)$ the iterates
converge to the unique saddle point, geometrically with factor
$q = 1 - 2\lambda/(\norm{G}_2^2 + \lambda)$~\cite{Tseng1995,Mokhtari2020}.
\end{theorem}

The result is the standard one for a monotone operator and we do not
reproduce its Lyapunov argument~\cite{Tseng1995,Facchinei2003,Mokhtari2020},
because its hypothesis does \emph{not} hold in the operative regime
($\lambda = 10^{-3}$ against $\sigma_{\max}(G_E)^2 \approx 9.4 \times
10^{7}$).  The margin is wide enough to make the rate vacuous,
$q \approx 1 - 2 \times 10^{-11}$, so the observed termination at $141$
iterations is set by the relative-value tolerance
$\varepsilon_{\mathrm{SP}} = 10^{-4}$ and not by any contraction.
\Cref{thm:convergence} therefore does not certify the runs in this
paper, and we certify each solution after the fact instead.  Let
$(\bm{U}^{(P)\circ}, \bm{U}^{(E)\circ})$ be the returned iterate.  Its
\emph{pursuer-side best-response gap} is the amount by which the
pursuer could improve the terminal distance against the returned
evader,
\begin{equation}\label{eq:br_gap}
\gamma_P = \norm{\bm{\delta}(\bm{U}^{(P)\circ}, \bm{U}^{(E)\circ})}_2
- \min_{\bm{U}^{(P)} \in \ctrlset_P^N}
\norm{\bm{\delta}(\bm{U}^{(P)}, \bm{U}^{(E)\circ})}_2,
\end{equation}
a convex program.  A small $\gamma_P$ certifies that the pursuer plays
near-optimally against the evader the solver produced.  Across the
engagements of \cref{sec:case_a,sec:case_b,sec:case_c} we find
$\gamma_P \leq 2.9$~\si{\meter}, and under half a metre in the two
one-on-one cases; where a stopping tolerance tuned at one horizon is
carried to a much longer one the gap grows to tens of metres, which
\cref{sec:scope_checks} reports rather than hides, since exposing
exactly that is what a per-run check is for.
The evader-side gap is generally large: as \cref{prop:curvature}
predicts, the evader could unilaterally do much better against a
\emph{fixed} pursuer point.  The returned iterate is therefore a
strong pursuer strategy together with a plausible evader reply, and it
coincides with the iterated-best-response outcome to within $0.5\%$
(\cref{sec:case_a}).

\subsection{Multiple Pursuers}\label{sec:multipursuer}

With $n_P > 1$ pursuers the inner $\min_i$ in \cref{eq:G1} is
non-smooth, and the exact convex handling is an epigraph reformulation
with a scalar slack for the inner minimum.  We instead use the
smoother and cheaper weighted surrogate
\begin{equation}\label{eq:multi_payoff}
J_{\mathrm{multi}} = \sum_i w_i\, \norm{\pos_N^{(P_i)} - \pos_N^{(E)}}_2^2
+ \lambda \sum_i \norm{\bm{U}^{(P_i)}}_2^2
- \lambda \norm{\bm{U}^{(E)}}_2^2,
\end{equation}
with adaptive weights $w_i \propto \exp(-\norm{\bm{\delta}_i}_2^2 /
\min_j \norm{\bm{\delta}_j}_2^2)$ that concentrate effort on the
pursuer with the most favorable geometry.  This is a coordination
heuristic and we do not claim it solves \cref{eq:G1}.  What we
\emph{can} state rigorously about the multi-pursuer engagement is the
joint escape certificate \cref{eq:phi_joint}: whenever
$\phi_N^{\mathrm{joint}} \geq 0$, no single evader direction clears
every pursuer, and the target cannot guarantee escape even though it
may escape any one pursuer alone.  \Cref{sec:case_c} exhibits exactly
this: two pursuers each individually escapable
($\phi_N^{(i)} < 0$) but jointly inescapable
($\phi_N^{\mathrm{joint}} > 0$).

\subsection{Keep-Out Zones by Successive Convexification}\label{sec:obstacle}

Keep-out balls $\norm{\pos_k - \pos_m^{\mathrm{obs}}}_2 \geq r_m$ are
nonconvex, and adding a raw quadratic penalty to $F$ destroys the
reachable-set geometry the method depends on: the penalty gradient at
the extragradient midpoint can dominate the pursuit gradient and drive
the iteration to a non-equilibrium point (we return to this failure in
\cref{sec:case_b}).  We instead convexify successively.  At each outer
iteration the previous trajectory supplies, for each active step $k$
and obstacle $m$, a supporting half-space with outward normal
$\bm{\eta}_{km} = (\pos_k^{\mathrm{prev}} - \pos_m^{\mathrm{obs}}) /
\norm{\pos_k^{\mathrm{prev}} - \pos_m^{\mathrm{obs}}}$, and the penalty
is applied to the \emph{affine} margin,
\begin{equation}\label{eq:scp_penalty}
P(\bm{U}) = \kappa \sum_{k,m}
\max\!\bigl(0,\; r_m - \bm{\eta}_{km}^\top(\pos_k - \pos_m^{\mathrm{obs}})\bigr)^2 .
\end{equation}
This term is convex in each player's own controls, so the inner game
retains the structure of \cref{sec:extragradient}; a few
relinearizations, followed by one hard-constrained polishing solve,
drive the keep-out violation to zero.  A penalty weight $\kappa = 10^3$
gives exact feasibility in every scenario tested, with the violation
decreasing as $O(1/\kappa)$ before the polish (\cref{sec:case_b}).

Two alternatives deserve mention, since keep-out avoidance is the one
genuinely nonconvex ingredient in the paper.  A mixed-integer
formulation encodes the disjunction exactly and returns a global
optimum~\cite{Richards2002}, at a branch-and-bound cost that grows with
the number of zone--step pairs and rules out a solve at the timescale
sought here.  Control barrier functions give forward invariance of the
safe set and compose cleanly with quadratic-program
controllers~\cite{Breeden2023RobustCBF,vanWijk2024DCBF}, but they
presuppose a feedback law and a valid barrier certificate for the
two-player dynamics, neither of which an open-loop game supplies.  We
choose successive convexification because it is the alternative that
preserves what the rest of the paper rests on: every subproblem stays
convex in each player's own controls, so the projection stays
element-wise, and the certificates, computed on the box and indifferent
to the zones, stay exact.

\section{Numerical Experiments}\label{sec:experiments}

All circular-orbit experiments use the planar 4D HCW model with
box-bounded controls at LEO altitude 500~km
($n = 1.131 \times 10^{-3}$~rad/s), timestep $\Delta t = 10$~s,
horizon $N = 30$ ($300$~s), thrust bounds $\bar{u}_P = 10$~mm/s$^2$
and $\bar{u}_E = 5$~mm/s$^2$ (2:1 ratio), and capture radius
$r_{\mathrm{cap}} = 50$~m.  Full parameters are in \cref{app:params}.

\subsection{Setup and Baselines}\label{sec:setup}

Two baselines accompany the extragradient solver
(\cref{sec:extragradient}).  \emph{Nominal MPC} assumes a passive
target ($\ctrl^{(E)} \equiv 0$) and solves the single-agent QP via
CVXPY+OSQP~\cite{CVXPY2016,OSQP2020}; its miss distance is a lower
bound that measures the cost of assuming passivity.  \emph{Iterative
best response} (IBR) alternates pursuer and evader QPs; because the
evader subproblem is linearized, IBR solves a related but distinct
game, so its agreement with the extragradient value is a
cross-check, not a proof.  Pseudocode is \cref{alg:ibr} in
\cref{app:algorithms}.

Neither is a competitor.  That nominal MPC captures
a passive target is a foregone conclusion; its role is to price the
passivity assumption, in metres and in $\Delta v$
(\cref{tab:case_a,sec:solver_char}), which is the argument for modeling
the engagement as a game at all.  IBR is an independent
implementation reaching the same values through different
subproblems, which is what makes an agreement informative.

\subsection{Case A: One-on-One Open-Loop Game}\label{sec:case_a}

The inspector starts at $\state_0^{(P)} = [200, -300, 0, 0]^\top$~m
and the target at the origin.  \Cref{tab:case_a} collects the results.
Against a passive target the inspector closes completely; nominal MPC
reaches the capture radius at $t = 280$~\si{\second}
($d = 40.9$~\si{\meter} at first passage).  Once the target thrusts to
escape, both game solvers finish at $d_f \approx 72.5$~\si{\meter},
outside the $50$~\si{\meter} radius; the extragradient value
($72.63$~\si{\meter}) and IBR ($72.28$~\si{\meter}) agree to $0.5\%$,
and the pursuer-side best-response gap is $\gamma_P =
0.40$~\si{\meter}, so the returned pursuer strategy is certifiably
near-optimal against the returned evader.  The escape certificate is
consistent with this: $\phi_N = -22.3$~\si{\meter}, guaranteeing the
target a standoff of at least $72.3$~\si{\meter}, and the pursuer
security value is $\bar{V} = 319.7$~\si{\meter}, so the outcome is
bracketed in $[72.3, 319.7]$~\si{\meter}, comfortably outside capture.
\Cref{rem:existence} sets this against the weaker question of whether a
close pair of plans merely exists, which for this geometry it does.
The gap between the two bounds is wide because the security value
against an unconstrained evader reachable set is inevitably loose at a
single horizon; it is the escape certificate that decides this
engagement.  \Cref{fig:case_a_traj} shows the resulting trajectories:
both agents run along nearly the same line, down-radial and
along-track, with the target holding its lead to the final sample
instead of turning -- the certifying direction $\bm{d}^*$ of
\cref{thm:escape} is what it is running along.  Both game solutions are
bang-bang, as the
linear dynamics and terminal payoff require, and
\cref{fig:control_profile} shows it directly: each control component sits on its bound at $29$ of the $30$
steps ($96.7\%$ of components for both agents), the only interior values
being those of the final step, where the effort term $\lambda$ is the
only thing acting.  Neither plan changes sign, so the optimal open-loop
strategies are constant-direction saturated burns -- the same sign
pattern that attains the support value in the certificate proof of
\cref{thm:escape}.  The structure is a property of the unconstrained
terminal-distance game rather than of the solver: once a keep-out
constraint binds in Case~B below, the inspector's control leaves the box
vertices and saturation drops to $65\%$.

\begin{table}[htbp]
\centering\small
\caption{Case~A (one-on-one open loop): terminal miss distance $d_f$,
  minimum separation $d_{\min}$, per-agent $\Delta v$, maximum relative
  speed, solve time, and iterations.  $N = 30$, $\Delta t = 10$~s,
  $r_{\mathrm{cap}} = 50$~m, $2{:}1$ thrust ratio.  Nominal MPC reports
  first-passage capture.}
\label{tab:case_a}
\begin{tabular}{@{}lrrrrrrr@{}}
\toprule
Method & $d_f$ (m) & $d_{\min}$ (m) & $\Delta v_P$ & $\Delta v_E$ & $v_{\mathrm{rel}}^{\max}$ & Time (s) & Iter \\
\midrule
Nominal MPC   & 40.9\,$^\dagger$ & 40.9 & 2.20 & 0.00 & 2.04 & 0.08 & -- \\
IBR           & 72.3 & 72.3 & 4.24 & 2.12 & 1.96 & 1.38 & 3 \\
Extragradient & 72.6 & 72.6 & 4.16 & 2.11 & 1.90 & 0.05 & 141 \\
\bottomrule
\end{tabular}
\\[2pt]{\footnotesize $^\dagger$ Captured at $t_{\mathrm{cap}} = 280$~s
  (first passage).  Times are single cold-start calls including matrix
  setup, as logged by the reproducibility scripts; \cref{tab:timing}
  reports medians over $20$ repeats
  ($23.5$~\si{\milli\second} for the extragradient solver), which are the
  figures quoted in the text.}
\end{table}

\Cref{fig:feasibility} sweeps $\phi_N$ over horizon at four thrust
ratios.  At the $2{:}1$ baseline it stays negative (escape certified,
matching $d_f \approx 72$~\si{\meter}); it crosses zero near a
$2.1{:}1$ ratio, above which the certificate falls silent and the game
must be solved.  Notably the security bound $\bar{V}$ does not fall to
the capture radius at any ratio for this geometry, for the reason given
above, which is one motivation for the receding-horizon play of
\cref{sec:tradespace}.

\begin{figure}[!htbp]
\centering
\includegraphics[width=0.458\textwidth]{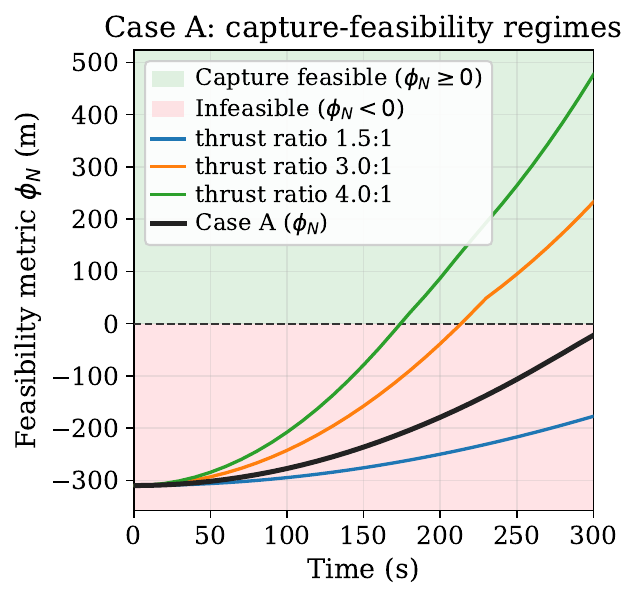}
\caption{Escape certificate $\phi_N$ versus horizon for Case~A at
  thrust ratios $1.5{:}1$, $2{:}1$, $3{:}1$, and $4{:}1$.  Negative
  $\phi_N$ (red band) certifies escape; the $2{:}1$ baseline stays
  below zero, consistent with \cref{tab:case_a}, while higher ratios
  cross into the region where the game must be solved.}
\label{fig:feasibility}
\end{figure}

\begin{figure}[!htbp]
\centering
\includegraphics[width=0.911\textwidth]{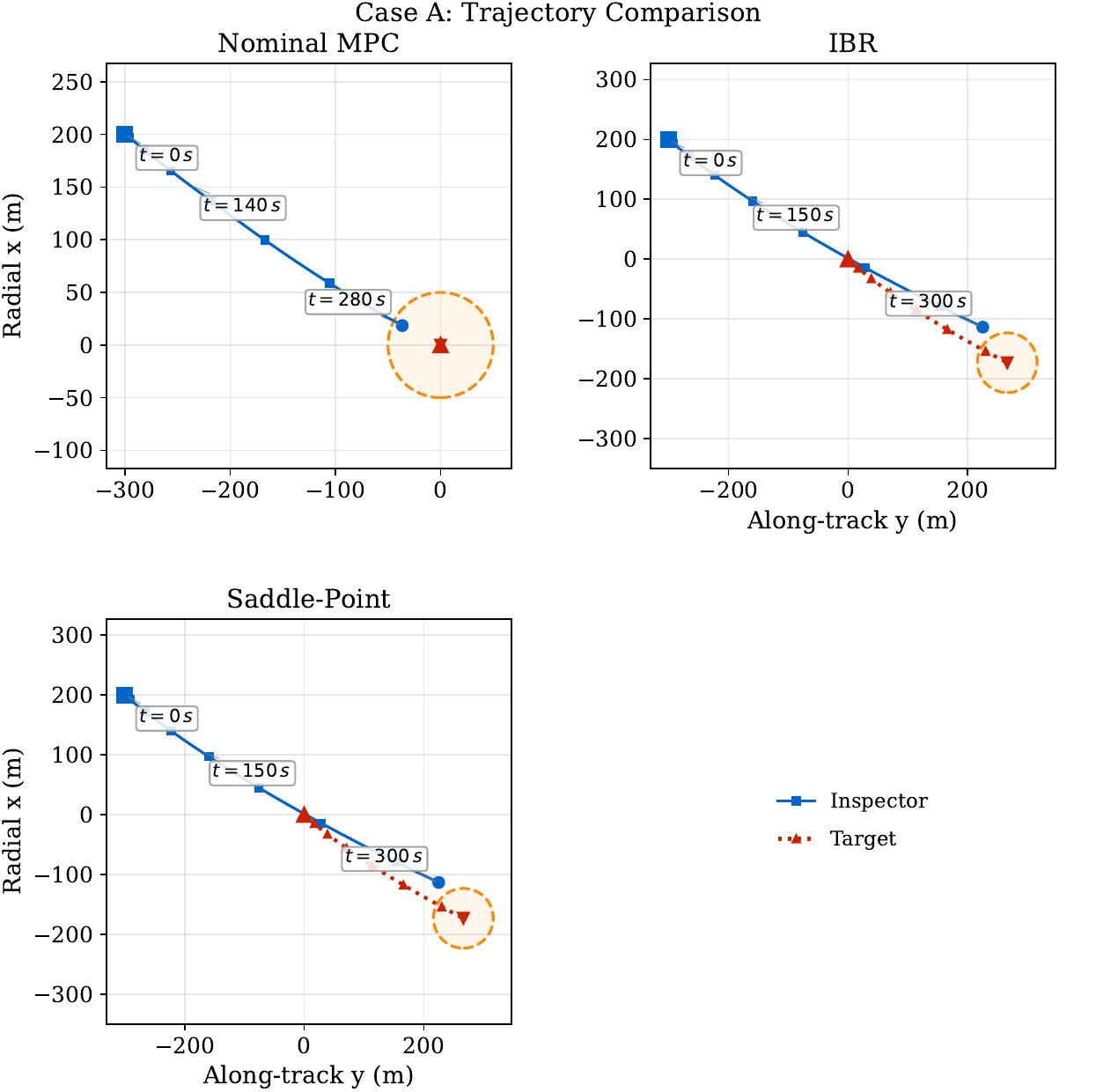}
\caption{Case~A trajectories in the LVLH $x$--$y$ plane.  Solid blue:
  inspector; dotted red: target; dashed circle: capture zone.
  Markers show the inspector at the start, mid-trajectory, and
  terminal sample.
  Panels are drawn at the same size, so the metres-per-point scale differs
  between them; within a panel, distance reads alike along both axes.}
\label{fig:case_a_traj}
\end{figure}

\begin{figure}[!htbp]
\centering
\includegraphics[width=0.823\textwidth]{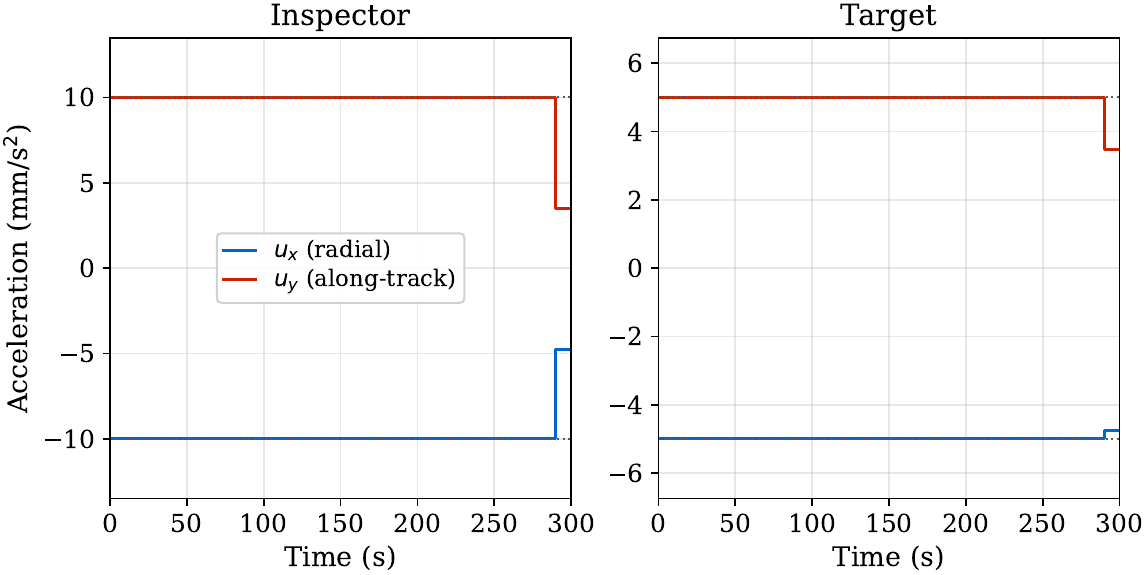}
\caption{Control histories of the Case~A extragradient solution, drawn
  as the zero-order-hold staircases they are.  Dotted lines mark the box
  bounds ($\pm 10$ and $\pm 5$~\si{\milli\meter\per\second\squared}).
  Every component of both plans is saturated except at the final step,
  and no component changes sign: the solutions are bang-bang with a
  constant thrust direction.}
\label{fig:control_profile}
\end{figure}

\subsection{Case B: Engagement with Keep-Out Zones}\label{sec:case_b}

The inspector starts at $[300, -400, 0, 0]^\top$~m with three keep-out
balls between the agents (\cref{app:params}).  This case both
demonstrates the successive-convexification scheme of
\cref{sec:obstacle} and shows why it is needed.  A naive quadratic
penalty added directly to $F$ fails badly here: evaluated at the
extragradient midpoint, where the extrapolated pursuer trajectory dips
into a keep-out ball, the penalty gradient is nearly twenty times the
pursuit gradient and drives the iteration to a spurious point at which
the \emph{pursuer flees the target}, reporting a physically
meaningless $d_f = 964.6$~\si{\meter}; a best-response check exposes it
at once, as a feasible pursuer can cut that distance to
$305$~\si{\meter}.  The successive-convexification scheme
\cref{eq:scp_penalty} avoids this by keeping every subproblem convex.
It converges to $d_f = 306.3$~\si{\meter} with zero keep-out violation
after a hard-constrained polish, and its pursuer-side best-response gap
against the corrected evader is $0.4$~\si{\meter}
(\cref{tab:case_b}).  Exact feasibility is not free.  The outer loop and
its hard-constrained polish take $67$~\si{\second}, against
$25$~\si{\milli\second} for the unconstrained game, so the keep-out
variant is a ground-planning or slow-loop tool rather than the onboard
one; this is the same trade the treatment of exclusion constraints in
\cref{sec:limitations} predicts, since the one-shot solve is what
nonconvexity costs.  The certificates are unaffected, being obstacle-free
constructions: for this geometry $\phi_N = -177.7$~\si{\meter}, so the
target holds at least $227.7$~\si{\meter} against any inspector plan,
constrained or not.  The keep-out zones cost the inspector a
$78$~\si{\meter} detour relative to the unconstrained game
($228.1$~\si{\meter}); the inspector, not the target, is the one
constrained by the obstacles: it skims the tightest ball to within
$10^{-7}$~\si{\meter} at $t = 240$~\si{\second}
(\cref{fig:case_b_traj}), while the target never comes within
$81$~\si{\meter} of any zone.  The unconstrained IBR value
($227.8$~\si{\meter}, obtained by driving straight through a ball with
a $49$~\si{\meter} violation) matches the unconstrained extragradient
value ($228.1$~\si{\meter}) and is listed only as a reference for the
detour cost.  A penalty-weight and obstacle-geometry study
(\cref{sec:obstacle_sensitivity}) confirms exact feasibility across
$\kappa$ and four geometries, including a threaded ten-metre corridor.

\begin{table}[htbp]
\centering\footnotesize
\caption{Case~B (keep-out zones).  The successive-convexification
  extragradient solver enforces the constraints exactly; the
  unconstrained MPC and IBR rows are references (they drive through the
  zones).  ``Clr.''\ is the signed minimum clearance to any keep-out
  boundary (negative $=$ violation).  Path metrics are reported as in
  \cref{tab:case_a}; here the approach is monotone, so $d_{\min}$
  coincides with $d_f$ in every row.}
\label{tab:case_b}
\resizebox{\textwidth}{!}{%
\begin{tabular}{@{}lrrrrrrrl@{}}
\toprule
Method & $d_f$ (m) & $d_{\min}$ (m) & Clr.\ (m) & $\Delta v_P$ & $\Delta v_E$ & $v_{\mathrm{rel}}^{\max}$ & Time (s) & Keep-out \\
\midrule
Nominal MPC        & 34.0$^\dagger$ & 34.0 & $-22.7$ & 3.69 & 0.00 & 3.39 & 0.06 & ignored \\
IBR                & 227.8 & 227.8 & $-49.1$ & 4.24 & 2.12 & 1.89 & 1.32 & ignored \\
Extragradient (no zones) & 228.1 & 228.1 & $-49.1$ & 4.16 & 2.11 & 1.83 & 0.03 & ignored \\
Successive convex.\ (+ polish) & 306.3 & 306.3 & $+0.0$ & 4.20 & 2.12 & 1.65 & 67.5 & enforced \\
\bottomrule
\end{tabular}%
}
\\[2pt]{\footnotesize $^\dagger$ Captured at $t_{\mathrm{cap}} = 290$~s (first passage).}
\end{table}

\begin{figure}[!htbp]
\centering
\includegraphics[width=0.911\textwidth]{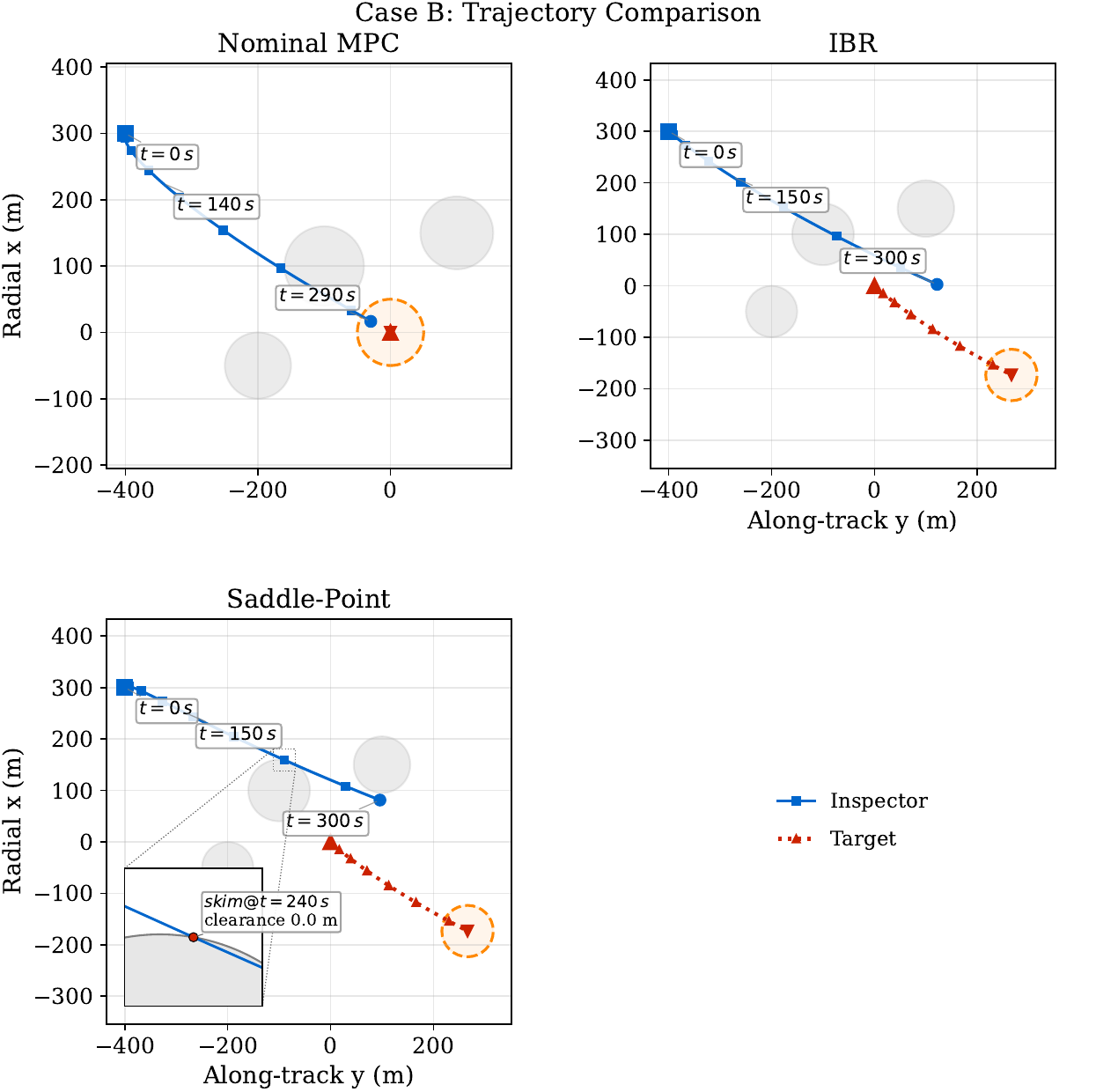}
\caption{Case~B trajectories with three keep-out balls (grey).  Under
  successive convexification (right) the inspector detours around the
  zones while the target flees; the inset zooms on the tightest
  clearance, where the inspector skims the keep-out boundary at
  $t = 240$~\si{\second}.  The passive-target MPC panel terminates at
  first-passage capture.
  Panels are drawn at the same size, so the metres-per-point scale differs
  between them; within a panel, distance reads alike along both axes.}
\label{fig:case_b_traj}
\end{figure}

\subsection{Case C: Two Pursuers and the Value of Encirclement}\label{sec:case_c}

Two inspectors at $\state_0^{(P_1)} = [300, -200, 0, 0]^\top$~m and
$\state_0^{(P_2)} = [-200, -350, 0, 0]^\top$~m engage a target at the
origin.  This case is where the joint escape certificate earns its
place.  Taken one at a time, each pursuer is escapable:
$\phi_N^{(1)} = -98.6$~\si{\meter} and $\phi_N^{(2)} =
-103.9$~\si{\meter}, so against either alone the target can guarantee a
large standoff.  Taken together they are not: no single evasion
direction clears both reachable sets, and the joint certificate is
$\phi_N^{\mathrm{joint}} = +70.0$~\si{\meter} $> 0$.  Encirclement thus
destroys the guarantee even though neither pursuer breaks it
individually, which is a rigorous statement of why a second inspector
helps, independent of any solver.  The weighted surrogate
(\cref{sec:multipursuer}) plays this out, closing the target to
$d_f = 136.8$~\si{\meter} as it trades distance against one pursuer for
distance against the other (\cref{fig:case_c_traj}).  The pursuer-side
best-response gaps here are $1.9$ and $2.8$~\si{\meter} for the two
inspectors, larger than the sub-metre gaps of Cases~A and~B.  That is
the expected price of the surrogate rather than a solver failure: it
coordinates the team but does not solve \cref{eq:G1}, and the rigorous
multi-pursuer statement remains the joint certificate above.

The sequential multi-pursuer IBR baseline is instructive as a failure.
It does not converge: both pursuers best-respond to the same stale
target prediction, so it enters a two-cycle (alternating $52.6$ and
$191.5$~\si{\meter}) that more iterations do not settle, and damping
merely lengthens the cycle.  Sequential best response has no fixed
point on this instance (\cref{tab:case_c2}), which is a concrete
reason to prefer the simultaneous weighted update.

\begin{table}[htbp]
\centering\footnotesize
\caption{Case~C (two pursuers versus one target).  $d_f$ is the closest
  inspector--target terminal distance and $\Delta v_P$ the sum over both
  inspectors ($3.94 + 3.78$~\si{\meter\per\second}).  Sequential IBR
  does not converge (two-cycle), so it has no single trajectory on which
  to measure path metrics; the extragradient surrogate settles at
  $135.2$~\si{\meter} given enough iterations, reported at the
  $200$-iteration cap as $136.8$~\si{\meter}.}
\label{tab:case_c2}
\begin{tabular}{@{}lrrrrrrl@{}}
\toprule
Method & $d_f$ (m) & $d_{\min}$ (m) & $\Delta v_P$ & $\Delta v_E$ & $v_{\mathrm{rel}}^{\max}$ & Time (s) & Status \\
\midrule
Sequential IBR & $52.6$--$191.5$ & -- & -- & -- & -- & 13.9 & two-cycle (no convergence) \\
Extragradient surrogate & 136.8 & 136.8 & 7.72 & 2.11 & 2.59 & 0.08 & $200$-iteration cap \\
\bottomrule
\end{tabular}
\end{table}

\begin{figure}[!htbp]
\centering
\includegraphics[width=0.911\textwidth]{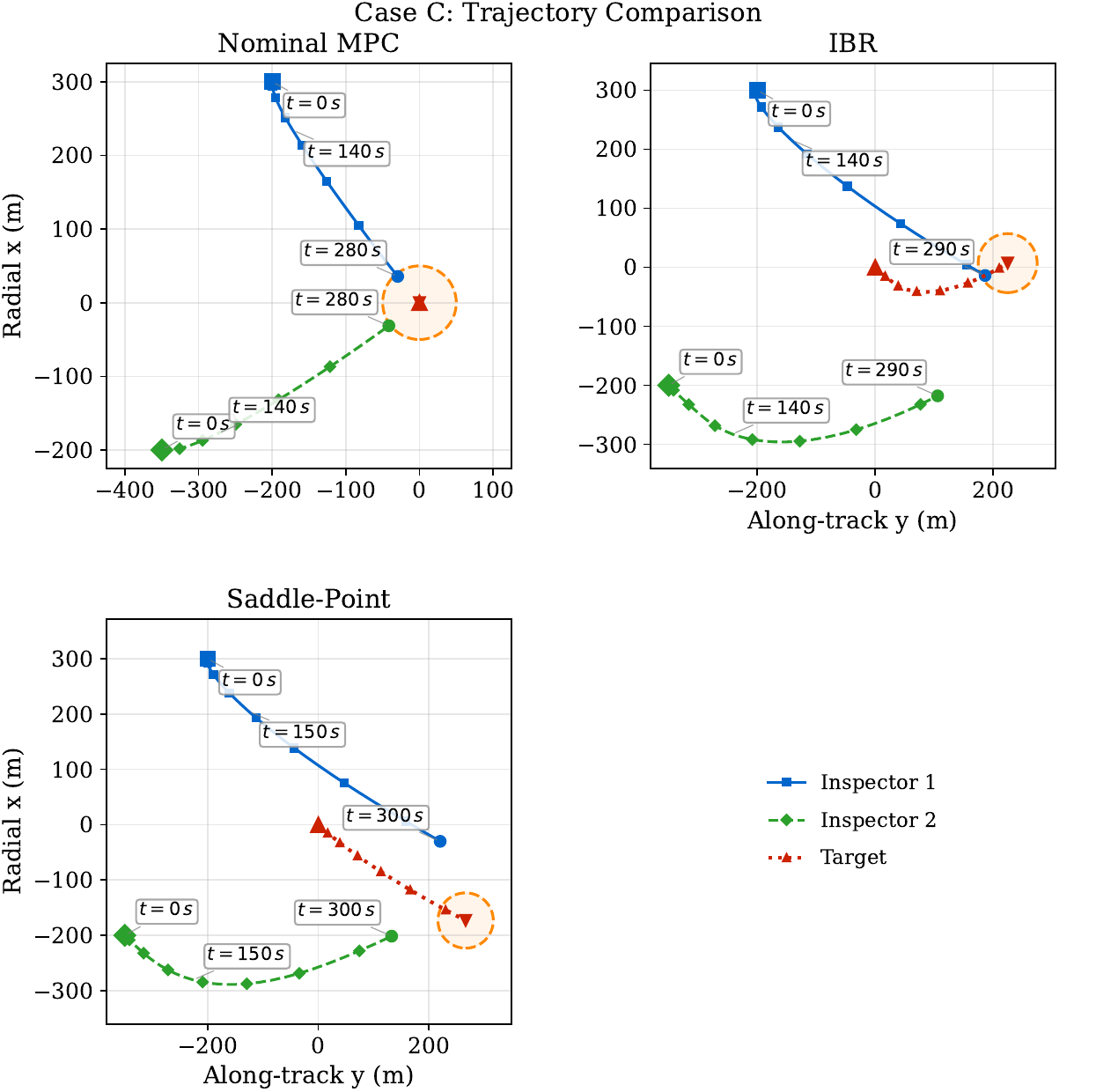}
\caption{Case~C (two pursuers versus one target).  The target trades
  distance against Inspector~1 for distance against Inspector~2,
  ending nearer than in the single-inspector case.  Markers show each
  inspector at the start, mid-trajectory, and terminal sample.  The three
  panels are drawn at the same size, so the metres-per-point scale differs
  between them; within a panel, distance reads alike along both axes.}
\label{fig:case_c_traj}
\end{figure}

\subsection{Penalty and Geometry Sensitivity of the Keep-Out Scheme}\label{sec:obstacle_sensitivity}

Sweeping the penalty weight $\kappa \in \{1, 10, 100,
10^3, 10^4\}$, the unpolished violation of the
successive-convexification scheme decays as $O(1/\kappa)$, from
$49$~\si{\meter} at $\kappa = 1$ to $0.04$~\si{\meter} at $\kappa =
10^4$, and the hard-constrained polish drives it to zero for any
$\kappa \geq 10$.  The polished terminal distance settles at
$306$~\si{\meter} for $\kappa \geq 100$ (the $\kappa = 10^3$ value used
throughout); at $\kappa = 10$ the scheme converges instead to a distinct
feasible trajectory at $255$~\si{\meter}, a reminder that the keep-out
problem is nonconvex and admits several feasible local solutions.  The
naive quadratic penalty, by contrast, has no usable weight: at
$\kappa = 1$ it ignores the constraint, and for $\kappa \geq 10$ it
produces the fleeing-pursuer artifact of \cref{sec:case_b}.  Across
four obstacle geometries (the zones shifted, enlarged by half,
blocking the line of sight, and arranged as a ten-metre corridor), the
scheme reaches zero violation in every case, threading the corridor at
a detour cost of only $7$~\si{\meter}.  The one failure is
$\kappa = 1$, where the initial iterate is too deeply infeasible for
the polish to recover; any $\kappa \geq 10$ avoids it.  This is the
evidence behind the exact-constraint-satisfaction entry of
\cref{tab:assumptions}.

\subsection{Engagement Tradespace}\label{sec:tradespace}

To map outcomes across the operational parameter space, the game is
wrapped in a receding-horizon loop and solved on a $15 \times 15$ grid
of thrust ratio versus warning time at fixed initial geometry
($d_0 = 128$~\si{\meter}), and on ten named scenarios spanning thrust
ratios $1.2{:}1$--$4{:}1$, initial separations $76$--$532$~\si{\meter},
and horizons $125$--$600$~\si{\second}.  The receding-horizon wrapper
is an empirical extension that re-solves the open-loop game at each
step and applies the first control; we claim no recursive-feasibility
proof for it, though per-step feasibility is checkable online from the
reachable-set screen.  Capture is the first-passage
event of \cref{sec:problem}: the engagement ends when the separation
first enters the capture ball.

\Cref{fig:heatmap} shows the feasibility map.  The frontier is
monotone in warning time: with enough warning, a thrust ratio as low
as $1.2{:}1$ suffices, and the minimum capturing ratio falls smoothly
from $4{:}1$ at $90$~\si{\second} to $1.2{:}1$ beyond
$280$~\si{\second}.  Escape, in the ten scenarios, is instead set by
the initial geometry: every escaping case begins at $d_0 \geq
200$~\si{\meter} with a target velocity head start that along-track
drift amplifies faster than the inspector can close.  The ten
trajectories are shown in \cref{fig:gallery}, truncated at the
first-passage event, so the figure and \cref{tab:tradespace} score
every scenario by the same criterion: the five captures reach the
capture circle, and the five escapes stay outside it for the whole
engagement.

\begin{figure}[!htbp]
\centering
\includegraphics[width=0.435\textwidth]{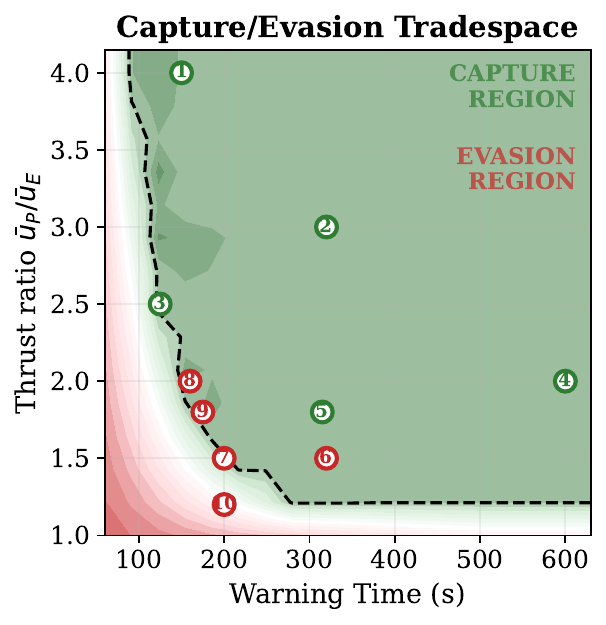}
\caption{Capture / escape feasibility map over thrust ratio and
  warning time at fixed initial geometry ($d_0 = 128$~\si{\meter}).
  Dashed line: the frontier separating capture (green) from escape
  (red).  Numbered markers: the ten scenarios of \cref{tab:tradespace};
  their colors reflect each scenario's outcome at its own initial
  separation, which can differ from the fixed-geometry background
  (e.g.\ S8 at $d_0 = 532$~\si{\meter}).}
\label{fig:heatmap}
\end{figure}

\begin{table}[!ht]
\centering\small
\caption{Ten-scenario tradespace under the receding-horizon game with
  first-passage capture.  Top block (S1--S5): capture, reported by
  time-to-capture $t_{\mathrm{cap}}$ and the separation at capture.
  Bottom block (S6--S10): escape, reported by the minimum separation
  reached (equal to the terminal separation, as these are still
  closing at the horizon).  Capture rate $5/10$; mean thrust ratio
  $2.66{:}1$ for capture, $1.60{:}1$ for escape.}
\label{tab:tradespace}
\resizebox{\textwidth}{!}{%
\begin{tabular}{@{}llrrrrrrrrrl@{}}
\toprule
ID & Name & Ratio & $d_0$ (m) & Horizon (s) & $r_{\mathrm{cap}}$ (m)
   & $t_{\mathrm{cap}}$ (s) & $d_{\min}$ (m) & $v_{\mathrm{rel}}^{\max}$ (m/s)
   & $\Delta V_P$ (m/s) & $\Delta V_E$ (m/s) & Outcome \\
\midrule
S1 & High thrust ratio          & 4.0 & 100 & 150 & 40 &  70 & 39.6 & 1.13 & 1.56 & 0.49 & \textbf{Capture} \\
S2 & Moderate, long horizon     & 3.0 & 144 & 320 & 40 & 168 & 37.0 & 0.90 & 2.39 & 1.16 & \textbf{Capture} \\
S3 & Close range, short horizon & 2.5 &  76 & 125 & 40 &  95 & 39.6 & 0.68 & 1.43 & 0.64 & \textbf{Capture} \\
S4 & Persistent low ratio       & 2.0 & 192 & 600 & 45 & 320 & 42.5 & 0.71 & 3.29 & 2.19 & \textbf{Capture} \\
S5 & Marginal approach          & 1.8 & 103 & 315 & 45 & 133 & 42.4 & 0.61 & 1.52 & 0.91 & \textbf{Capture} \\
\midrule
S6 & Along-track drift wins     & 1.5 & 292 & 320 & 50 & --  & 162.9 & 0.96 & 3.31 & 2.21 & Escape \\
S7 & Target velocity advantage  & 1.5 & 250 & 200 & 50 & --  & 208.8 & 0.74 & 2.04 & 1.36 & Escape \\
S8 & Too far at start           & 2.0 & 532 & 160 & 50 & --  & 456.1 & 0.94 & 2.15 & 1.07 & Escape \\
S9 & Narrow escape (target)     & 1.8 & 200 & 175 & 45 & --  & 135.9 & 0.92 & 2.14 & 1.19 & Escape \\
S10 & Orbit-mechanics escape    & 1.2 & 308 & 200 & 50 & --  & 281.4 & 0.32 & 1.63 & 1.36 & Escape \\
\bottomrule
\end{tabular}%
}
\end{table}

\begin{figure}[!htbp]
\centering
\includegraphics[width=\textwidth]{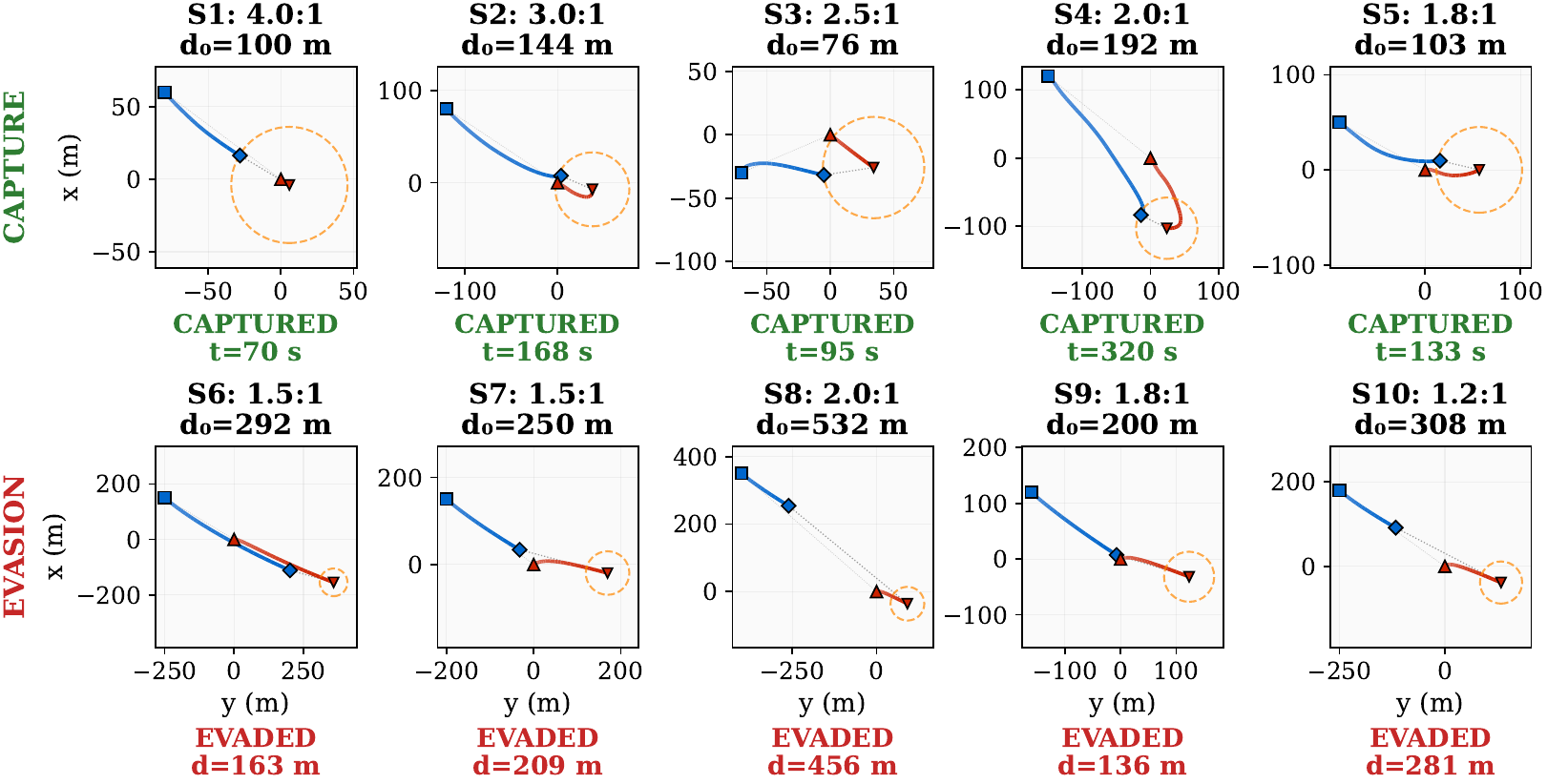}
\caption{Per-scenario trajectories under the receding-horizon game.
  Top row (S1--S5): capture, with trajectories truncated at the
  first-passage event (inspector marker on the capture circle).
  Bottom row (S6--S10): escape.  Blue: inspector; red: target; dashed
  circle: capture zone.  Every panel is drawn at the same size, so the
  metres-per-point scale differs between panels; within a panel, distance
  reads alike along both axes.}
\label{fig:gallery}
\end{figure}

\subsection{Design Implications}\label{sec:design_implications}

The tradespace supports a short list of design rules, with the
qualifications the data actually warrant.
\begin{enumerate}[leftmargin=*]
\item \emph{Thrust ratio is necessary but not sufficient.} Capture
  needs $\bar{u}_P/\bar{u}_E \geq 1.8$ (mean $2.66{:}1$ among
  captures); a $100$~\si{\kilo\gram} inspector with a $1$~\si{\newton}
  thruster ($\approx 10$~\si{\milli\meter\per\second\squared}) meets
  this against a target of half its authority.  Because thrust ratio
  reduces to $(T_P/m_P)/(T_E/m_E)$, a light, high-thrust-to-mass
  inspector outperforms a heavier one of nominally larger thrust.
\item \emph{Warning time buys thrust, monotonically and without
  reversal.} Warning time can be traded against installed thrust across
  the whole range swept (\cref{sec:tradespace}).
\item \emph{Detection range can dominate thrust.} S8 out-thrusts S3 by
  a smaller margin than S3 out-ranges it, and S3 is the one that
  captures, so acquisition range is the first thing to buy.  The
  $76$--$532$~\si{\meter} range and $90$--$600$~\si{\second} warning
  swept here are within demonstrated practice: relative GNSS covers this
  regime when the target cooperates~\cite{DAmico2012PRISMA}, and
  angles-only optical tracking of a non-cooperative target has been
  flown from tens of kilometres inward~\cite{Gaias2018AVANTI}, so at
  these separations warning time is set by the tracking update rate
  rather than by acquisition.
\item \emph{Budget.} A $5$~\si{\meter\per\second} reserve covers every
  observed capture at the $50$~\si{\meter} radius
  (\cref{tab:tradespace}).
\end{enumerate}
The capture radius itself is a design choice, and at $50$~\si{\meter}
it is an approach gate of the scale standard practice reserves around a
resident space object~\cite{Fehse2003} rather than a docking tolerance,
so we sweep it.  At
$r_{\mathrm{cap}} = 50$ and $20$~\si{\meter} the same five scenarios
capture; at $10$ and $5$~\si{\meter} the count drops to four, as S3's
short $125$~\si{\second} window no longer suffices, and time-to-capture
roughly doubles as the radius tightens.  The $1.8{:}1$ ratio threshold
is unchanged from $50$ down to $5$~\si{\meter}; what a tighter
docking-class radius demands is not more thrust but a longer
engagement window.

\subsection{Statistics, Screening, and Navigation Error}\label{sec:solver_char}

A $200$-trial Monte Carlo study perturbs the initial state
(inspector $\pm 50$~\si{\meter}, $\pm 0.5$~\si{\meter\per\second};
target $\pm 20$~\si{\meter}, $\pm 0.2$~\si{\meter\per\second}) and
records both game solvers and the certificate
(\cref{app:mc}).  Against a passive target the inspector captures in
$99.0\%$ of trials ($95\%$ Wilson interval $[96.4, 99.7]$); against a
maneuvering target the capture rate is $25.0\%$ $[19.5, 31.4]$ for both
the extragradient and IBR solvers, which agree on the outcome trial by
trial.  Capture concentrates at short range: $44\%$ for
$d_0 < 320$~\si{\meter}, falling to $5\%$ beyond
$400$~\si{\meter}.  The terminal-distance distribution is sharply
bimodal, with a capture cluster near $42$~\si{\meter} and an escape
cluster near $162$~\si{\meter} (\cref{fig:mc_histogram}).

\begin{figure}[!htbp]
\centering
\includegraphics[width=0.464\textwidth]{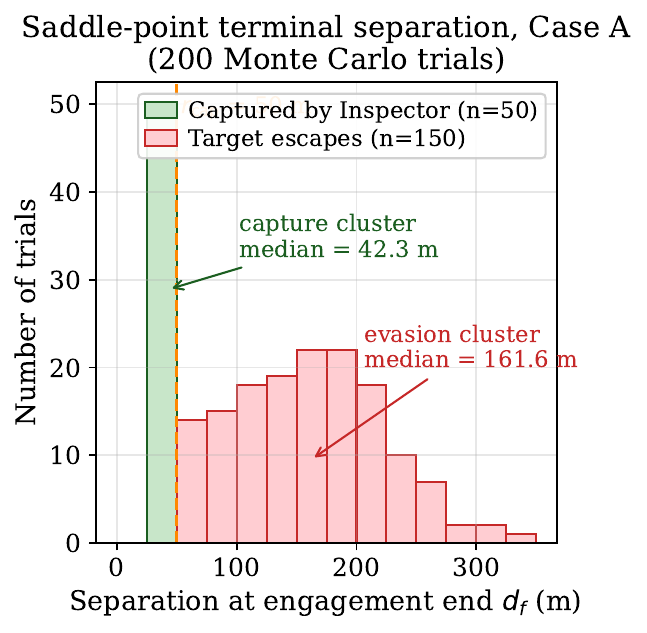}
\caption{Terminal-distance distribution over the $200$-trial Monte
  Carlo study under game play.  The distribution is bimodal: a capture
  cluster (green, median $42$~\si{\meter}) and an escape cluster (red,
  median $162$~\si{\meter}), split cleanly at the $r_{\mathrm{cap}} =
  50$~\si{\meter} radius (dashed).  The sign of the escape certificate
  $\phi_N$ predicts which cluster each trial lands in with no error.}
\label{fig:mc_histogram}
\end{figure}

The escape certificate performs better than its own statement requires,
and it is worth being precise about what is being measured.
\Cref{thm:escape} is one-sided: it says nothing when $\phi_N \geq 0$, so
it cannot be wrong on that branch.  The false-positive rate below
therefore scores $\phi_N$ used as a two-sided \emph{screen}, which is
more than the theorem claims, and is the quantity an operator deciding
whether to solve the full game actually cares about.  Used that way,
across all $200$ trials the
sign of $\phi_N$ predicts the game outcome without a single error
($0$ false positives out of $50$ predicted captures, $0$ false
negatives out of $150$ certified escapes; $95\%$ upper bounds $5.8\%$
and $2.0\%$).  On the escape branch the realized separation tracks the
certificate bound almost exactly, $d_f \approx r_{\mathrm{cap}} -
\phi_N$ with correlation $-0.9999$, so the bound is not merely
correct but tight, and the outer approximation costs little
conservatism in practice.  The perturbations leave the two branches
well separated, so this sample does not probe $\phi_N$ within a few
metres of zero; the certificate direction is safe there by
\cref{thm:escape}, but the screen's behaviour on the boundary is not
measured.

Perfect state knowledge is an idealization, so we also perturb it.  A
separate hundred-trial study plans open-loop from an \emph{estimated}
relative state and flies the plan against the truth, holding the true
initial conditions fixed so that navigation error is isolated from the
randomized geometry above; its unperturbed capture rate is $27\%$.
Cooperative-grade navigation (position
$0.1$~\si{\meter}, velocity $1$~\si{\milli\meter\per\second}, the
regime of GNSS relative navigation~\cite{DAmico2012PRISMA}) leaves
every outcome unchanged: the median terminal-distance shift is under
$0.01$~\si{\meter}, and no trial changes capture status.
Angles-only navigation of a non-cooperative target is far coarser, at
metre-level cross-range and roughly ten percent of range down the
boresight~\cite{Ardaens2018AnglesOnly,Gaias2018AVANTI}, and it cuts
that rate to $11\%$, with every flip a capture lost
to an escape.  Navigation error is thus one-sided: it penalizes the
inspector, which must close precisely, and never the target.  This
places the value of the whole framework in context: the
terminal-distance game and its certificates are only as good as the relative
navigation feeding them, and the binding operational requirement for a
non-cooperative inspection is sensing, not thrust.

The solver's efficiency is characterized in
\cref{app:scalability}.  In one shot the extragradient method returns
a strategy pair in about $25$~\si{\milli\second}, some forty times
faster than rebuilding the IBR QPs.  That comparison is cold-start only: a
condensed, warm-started QP implementation of IBR is faster still when
amortized over many solves.  The extragradient method's durable
advantages are that it requires no per-step problem compilation, its
memory is $O(Nn_u)$, and its cost scales linearly in the number of
pursuers.  The escape certificate costs a fraction of one game solve
once the reachable-set maps are in hand, so screening before solving
is nearly free.

\subsection{Model Checks: Out-of-Plane Motion and Long Engagements}\label{sec:scope_checks}

Two modeling choices carry the results so far: the planar reduction and
the $300$~\si{\second} horizon.  Both are checked here rather than
argued for.

\paragraph{Six states and out-of-plane motion}
The engagements of \cref{sec:case_a} were re-solved on the full
six-state HCW model of \cref{eq:hcw_3d} with three control axes,
keeping the in-plane initial conditions and the per-axis thrust bound
unchanged, so the inspector gains a third axis but no extra authority
per axis, which is what a body-mounted thruster set provides.
\Cref{tab:threed} collects the outcome.  With zero cross-track state
the six-state solve reproduces the planar terminal distance to
$4 \times 10^{-13}$~\si{\meter}, confirming the decoupling argument of
\cref{sec:dynamics} numerically rather than by inspection.  Cross-track
offsets up to $150$~\si{\meter} leave the outcome unchanged, because the
inspector can null that much separation within the horizon at
$10$~\si{\milli\meter\per\second\squared}; the cost appears as
$\Delta v$, which rises from $4.16$ to $4.62$~\si{\meter\per\second}.
At $300$~\si{\meter} of cross-track separation the offset can no longer
be nulled, the target holds $60.9$~\si{\meter} of it to the terminal
time, and the miss distance grows from $72.6$ to
$94.9$~\si{\meter}.  The escape certificate tracks this without
modification, tightening from $-22.3$ to $-44.3$~\si{\meter} and
bounding the realized separation to within $0.6$~\si{\meter}.  Three
effects are worth noting.  The certificate needs a direction template on
the sphere rather than the circle, and a template of practical size sits
far enough from the best direction to cost tens of metres of margin, so
the reported values refine the best template direction by local ascent;
this is sound because \cref{thm:escape} holds for every single
direction.  The security value is genuinely weaker in three dimensions,
rising from $319.7$ to $389.8$~\si{\meter}, because the target's
reachable set has grown an axis.  And the solver cost is essentially
unchanged, $19$--$22$~\si{\milli\second} against
$25$~\si{\milli\second} planar, since the extragradient iteration is
matrix--vector work in the control dimension.  Replaying the six-state
plans against nonlinear two-body-plus-$J_2$ propagation moves the
terminal distance by at most $0.030$~\si{\meter}.

The IBR baseline behaves instructively here: it returns
$72.3$~\si{\meter} in every configuration, including the
$300$~\si{\meter} one, because its linearized target subproblem does not
exploit the new axis.  That is the concrete form of the caveat in
\cref{sec:setup} that IBR solves a related but distinct game, and it is
why we treat its agreement as a cross-check rather than as a reference
value.

\begin{table}[htbp]
\centering\small
\caption{Case~A on the six-state model with out-of-plane initial
  offsets.  Per-axis thrust bounds, horizon, and in-plane initial
  conditions are those of \cref{tab:case_a}; $\Delta z_0$ is the initial
  cross-track separation and $\Delta z_N$ the terminal one.  The last
  column is the change in terminal distance when the same plans are
  flown against nonlinear ECI$+J_2$ dynamics.}
\label{tab:threed}
\begin{tabular}{@{}lrrrrrrr@{}}
\toprule
Geometry & $\Delta z_0$ & $\Delta z_N$ & $d_f$ & $\phi_N$ & $r_{\mathrm{cap}}-\phi_N$ & $\Delta v_P$ & Nonlin.\ shift \\
 & (m) & (m) & (m) & (m) & (m) & (m/s) & (m) \\
\midrule
Planar limit          &   0 &  0.0 & 72.63 & $-22.28$ & 72.28 & 4.16 & 0.024 \\
Cross-track offset    &  50 &  0.0 & 72.63 & $-22.28$ & 72.28 & 4.39 & 0.026 \\
Cross-track offset    & 150 &  0.0 & 72.63 & $-22.28$ & 72.28 & 4.62 & 0.030 \\
Cross-track offset    & 300 & 60.9 & 94.91 & $-44.34$ & 94.34 & 5.07 & 0.022 \\
Split $\pm150$        & 300 & 60.9 & 94.91 & $-44.34$ & 94.34 & 5.07 & 0.021 \\
Target drifting out-of-plane & 150 & 0.0 & 72.63 & $-22.28$ & 72.28 & 4.40 & 0.029 \\
\bottomrule
\end{tabular}
\end{table}

\paragraph{Longer engagements}
\Cref{tab:horizons} extends Case~A from $300$ to $1200$~\si{\second}.
The engagement changes character: at $300$~\si{\second} the target
escapes with the certificate confirming it, while from
$450$~\si{\second} onward the $2{:}1$ inspector closes to contact and
$\phi_N$ turns positive, which is the certificate correctly falling
silent rather than failing.  With enough time a thrust-advantaged
inspector wins, so the informative regime is the one where the horizon is
short enough for the outcome to be in doubt.

Two things do degrade with horizon, and both are visible in the table.
The first is a solver setting.  The stopping tolerance was tuned at
$N = 30$, and at $900$ and $1200$~\si{\second} it stops the iteration
early, leaving a pursuer best-response gap of $27.9$ and
$19.9$~\si{\meter} respectively -- large enough to matter, though not
large enough to change the capture label.  The per-run certificate is
what exposes this: tightening the tolerance to $10^{-8}$ closes the gap
to $2$~\si{\milli\meter} for $239$ and $608$ iterations, still
$135$ and $272$~\si{\milli\second}; a solver reporting only convergence
would have hidden it.  The second is the linear
model itself, which is the real limit on horizon: the peak position
discrepancy against nonlinear ECI$+J_2$ propagation grows from
$0.056$~\si{\meter} at $300$~\si{\second} to $2.56$~\si{\meter} at
$1200$~\si{\second}, a factor of $46$, while the terminal-distance
discrepancy stays under $0.35$~\si{\meter}, i.e.\ under $0.7\%$ of the
capture radius.  Engagements much beyond $1200$~\si{\second} would need
either re-linearization about a reference trajectory or the elliptical
LTV treatment of \cref{sec:elliptical}, not a different solver.

\begin{table}[htbp]
\centering\small
\caption{Case~A over horizons from $300$ to $1200$~\si{\second}
  ($\Delta t = 10$~\si{\second}).  $\gamma_P$ is the pursuer
  best-response gap at the default stopping tolerance and at
  $10^{-8}$, with the iteration count for the latter.  The last two
  columns compare against nonlinear ECI$+J_2$ propagation of the same
  plans.}
\label{tab:horizons}
\begin{tabular}{@{}rrlrrrrr@{}}
\toprule
$T$ & $d_f$ & Outcome & $\phi_N$ & $\gamma_P$ & $\gamma_P$ & Iter. & Peak lin. \\
(s) & (m) & & (m) & (m) & tight (m) & tight & error (m) \\
\midrule
 300 & 72.63 & escape  & $-22.3$ & 0.40  & 0.398 & 142 & 0.056 \\
 450 &  0.01 & capture & $244$   & 0.01  & 0.000 &  92 & 0.096 \\
 600 &  0.00 & capture & $660$   & 0.00  & 0.000 &  69 & 0.173 \\
 900 & 27.85 & capture & $2025$  & 27.85 & 0.002 & 239 & 0.258 \\
1200 & 19.89 & capture & $3782$  & 19.89 & 0.000 & 608 & 2.556 \\
\bottomrule
\end{tabular}
\end{table}

\section{Extension to Elliptical Reference Orbits}\label{sec:elliptical}

GTO, Molniya, and many high-altitude or science orbits carry
non-negligible eccentricity, for which time-invariant HCW is
inaccurate.  The Yamanaka--Ankersen (YA) state transition
matrix~\cite{Yamanaka2002} yields a linear time-varying (LTV) system,
which the solver takes without a single change, for the reason
\cref{rem:ltv_convexity} states precisely.

\subsection{Yamanaka--Ankersen LTV Model}\label{sec:ya_stm}

Relative motion about an elliptical Keplerian orbit of eccentricity
$e$ is governed by the Tschauner--Hempel
equations~\cite{Tschauner1965}, which reduce to HCW~\cref{eq:hcw_3d}
in the circular limit $e = 0$.  Yamanaka and
Ankersen~\cite{Yamanaka2002} constructed a closed-form state
transition matrix
\begin{equation}\label{eq:ya_stm}
\state(\nu) = \Phi(\nu,\nu_0)\,\state(\nu_0),
\end{equation}
parameterized by the true anomaly $\nu$ and resolved through the
Kepler equation $M = \mathcal{E} - e\sin \mathcal{E}$, in which
$\mathcal{E}$ is the eccentric anomaly, written so as not to collide
with the evader label $E$.  Its auxiliary integral is what
avoids the near-circular singularity discussed in
\cref{sec:related_elliptical}; the full derivation, including the
fundamental matrix entries and the time-to-anomaly mapping, is in
\cref{app:ya_stm}.

At each timestep $k$ the true anomaly advances by Kepler's equation
from $\nu_k$ to $\nu_{k+1}$, yielding the discrete LTV form
\begin{equation}\label{eq:ltv_discrete}
\state_{k+1} = A_k\,\state_k + B_k\,\ctrl_k,
\quad k = 0,\ldots,N-1,
\end{equation}
with $A_k = \Phi(\nu_{k+1},\nu_k)$ and $B_k$ obtained from Simpson's
rule quadrature of $\int_0^{\Delta t} \Phi(\nu_{k+1},\nu(\tau)) [0;\,I]\, d\tau$,
which achieves $O(\Delta t^5)$ accuracy.  At $e = 0$, $(A_k, B_k)$
reduce to the constant HCW matrices $(A_d, B_d)$ of
\cref{eq:discrete} to machine precision (see validation in
\cref{sec:hcw_recovery}).

\subsection{Reachable Sets and Game Formulation under LTV Dynamics}\label{sec:ltv_game}

The Minkowski-sum decomposition \cref{eq:reach_mink} and the
trajectory-space representation \cref{eq:traj_affine} both generalize
to LTV systems by replacing matrix powers with time-varying products
$\Phi_{N:j} = A_{N-1}A_{N-2}\cdots A_j$.  The forward reachable set
becomes
\begin{equation}\label{eq:ltv_reach_mink}
\reachset_N(\state_0) = \Phi_{N:0}\,\state_0 \oplus
\bigoplus_{k=0}^{N-1} \Phi_{N:k+1}\, B_k\, \ctrlset,
\end{equation}
with support-function evaluation
\begin{equation}\label{eq:ltv_sf_reach}
\support_{\reachset_N}(\bm{d}) =
\bm{d}^\top \Phi_{N:0}\,\state_0
+ \sum_{k=0}^{N-1} \support_{\ctrlset}\!\left(
  B_k^\top \Phi_{N:k+1}^\top \bm{d} \right);
\end{equation}
the LTV control-to-state mapping has block
$S_{k,j}^{\mathrm{LTV}} = \Phi_{k:j+1} B_j$ for $j<k$ in place of
$A_d^{k-j-1} B_d$.  Precomputing and caching the time-varying products
once leaves the $O(NL)$ screening cost and the $O(N n_p n_u)$
per-iteration solver cost unchanged.

\begin{remark}[Structure preserved under LTV dynamics]
\label{rem:ltv_convexity}
The LTV payoff has the same quadratic form as $J$ in \cref{eq:payoff}
with $G_P, G_E, \bm{g}_0$ replaced by their LTV counterparts.  Every
statement of \cref{sec:game} carries over verbatim: the curvature
threshold of \cref{prop:curvature} (with $G_E$ replaced by
$G_E^{\mathrm{LTV}}$), the escape certificate \cref{thm:escape}, and
the security bound \cref{def:security} all hold, since each depends on
the dynamics only through the terminal support functions
\cref{eq:ltv_sf_reach}.  The elliptical game is thus certified by the
same two inequalities as the circular one.
\end{remark}

\subsection{Eccentricity and Orbital Phase}\label{sec:ecc_results}

The mechanism behind the eccentricity dependence is geometric.  Near
periapsis the chief's transverse velocity is largest ($r v_\theta = h$
by angular-momentum conservation), so a deputy maneuvering relative to
it sees a strong, rapidly varying along-track shear; near apoapsis the
same deputy sees a gentle one.  The along-track drift that already
aids the target's escape in the circular case is therefore amplified
at periapsis and damped at apoapsis, and an engagement that begins at
the wrong orbital phase is harder for the inspector.

To keep the comparison physical, all elliptical scenarios fix a
$500$~\si{\kilo\meter} perigee altitude, so the semi-major axis grows
with eccentricity as $a = (R_\oplus + 500~\si{\kilo\meter})/(1-e)$.
Holding $a$ fixed instead would drive the perigee below the surface at
high $e$.  The Case~A scenario is re-run on the
LTV system over $e \in \{0, 0.1, 0.3, 0.6\}$ at periapsis and over
five initial true anomalies at the two higher eccentricities
(\cref{tab:eccentricity}).

\begin{table}[htbp]
\centering\footnotesize
\caption{Terminal distance and escape certificate at periapsis
  ($\nu_0 = 0^\circ$) for four eccentricities, fixed
  $500$~\si{\kilo\meter} perigee.  Other parameters as in Case~A.  Here
  $\phi_N$ is the $L=48$ screen value, so the $e=0$ row
  ($-22.2$~\si{\meter}) matches the $L=96$ certificate of
  \cref{sec:case_a} ($-22.3$~\si{\meter}) up to the direction-sampling
  resolution.  Five rows reach the $200$-iteration cap; re-solving them
  to $2\times10^{4}$ iterations moves $d_f$ by at most
  $0.21$~\si{\meter} and changes no certificate sign and no outcome.}
\label{tab:eccentricity}
\begin{tabular}{@{}rrrrrc@{}}
\toprule
$e$ & $a$ (km) & $\nu_0$ (deg) & $d_f$ (m) & $\phi_N$ (m) & Captured \\
\midrule
0.0 &  6871 &   0 & 72.6 & $-22.2$ & No \\
0.1 &  7634 &   0 & 75.9 & $-25.4$ & No \\
\midrule
0.3 &  9816 &   0 & 81.9 & $-31.8$ & No \\
    &       &  45 & 69.7 & $-19.3$ & No \\
    &       &  90 & 53.5 & $-3.3$  & No \\
    &       & 135 & 55.7 & $-5.4$  & No \\
    &       & 180 & 59.1 & $-9.1$  & No \\
\midrule
0.6 & 17178 &   0 & 91.1 & $-41.0$ & No \\
    &       &  45 & 69.2 & $-18.8$ & No \\
    &       &  90 & 53.9 & $-3.6$  & No \\
    &       & 135 & 67.2 & $-16.8$ & No \\
    &       & 180 & 71.3 & $-21.0$ & No \\
\bottomrule
\end{tabular}
\end{table}

Two patterns emerge.  Increasing eccentricity at periapsis raises the
miss distance modestly, from $72.6$ to $91.1$~\si{\meter} between
$e = 0$ and $e = 0.6$, a factor of $1.25$, and the certificate tracks
this in both sign and magnitude.  The dependence on orbital phase is
stronger: at
$e = 0.6$ the miss distance varies by a factor of $1.7$ across the
sweep ($91.1$~\si{\meter} at periapsis, $53.9$~\si{\meter} near
$\nu_0 = 90^\circ$), so \emph{when} in the orbit an engagement begins
matters as much as the thrust ratio does on a circular orbit.  Whether
the engagement is timed by wall-clock or by orbital arc changes the
story: holding the $300$~\si{\second} horizon fixed, escape remains
certified at every phase, but scaling the horizon to a fixed orbital
arc (about $1200$~\si{\second} at $e = 0.6$) gives the inspector enough
time that the certificate flips sign and capture becomes feasible at
every $\nu_0$.  The practical reading is that on eccentric orbits the
inspector should budget engagement time in orbital arc, not seconds.
\Cref{fig:eccentricity_traj} shows the geometry: as $e$ grows the
terminal reachable sets become asymmetric, radially squashed and
along-track elongated near periapsis, which is what widens the
target's escape margin $-\phi_N$ from $22.2$ to $41.0$~\si{\meter}
across the sweep.

Because $\nu_0$ is not known exactly in flight, where ephemeris and
timing error both surface as a phase offset, the outcome should not
depend sharply on it.  Planning on an assumed $\nu_0$ and flying the
resulting sequence against the truth, over $\nu_0 \in \{90^\circ,
135^\circ, 180^\circ\}$ at $e = 0.6$ and offsets up to $\pm
15^\circ$, moves the flown terminal distance by at most
$0.31$~\si{\meter} and changes no capture outcome.  The
\emph{predicted} value is far more sensitive than the flown one: at
$\nu_0 = 90^\circ$ a $+15^\circ$ error moves the planned $d_f$ from
$53.7$ to $58.3$~\si{\meter} while the flown result moves to
$54.0$~\si{\meter} (this sweep is run to convergence, which is the
$0.17$~\si{\meter} offset from the capped entry in
\cref{tab:eccentricity}), and it likewise inflates the certificate margin
from $-3.7$ to $-8.3$~\si{\meter} without changing its sign.  Anomaly
error of this size therefore costs accuracy in the prediction, not
correctness in the decision.  Larger or structured uncertainty is
better handled by inflating the reachable sets, as in the robust
extension noted in \cref{sec:discussion}.

\begin{figure}[!htbp]
\centering
\includegraphics[width=0.920\textwidth]{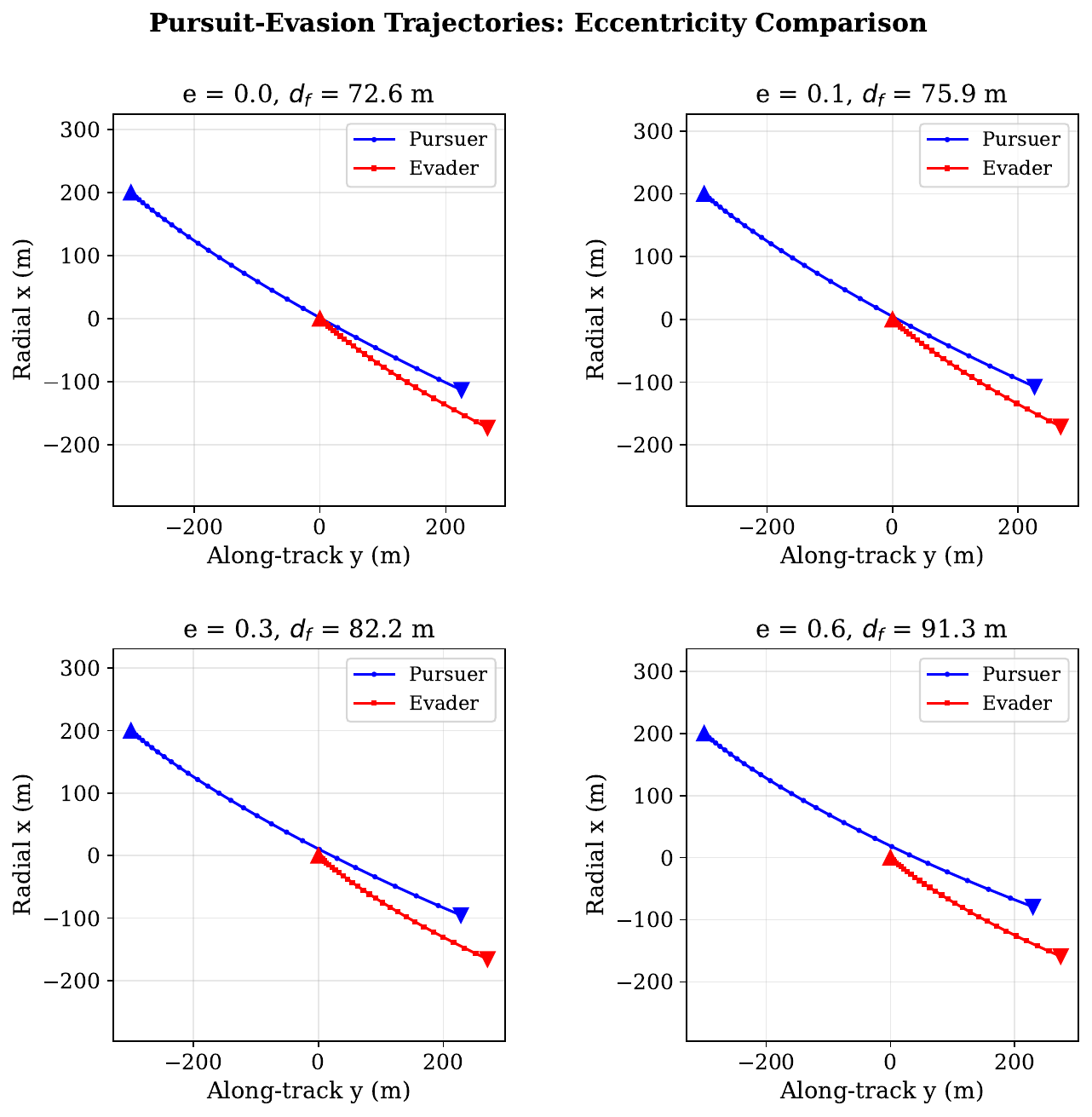}
\caption{Case~A trajectories at $e \in \{0, 0.1, 0.3, 0.6\}$ with
  $\nu_0 = 0$ (periapsis) and a fixed $500$~\si{\kilo\meter} perigee.
  Higher eccentricity modestly widens the target's along-track escape
  under the same thrust ratio and geometry.}
\label{fig:eccentricity_traj}
\end{figure}

\subsection{Validation Against Nonlinear Dynamics}\label{sec:hcw_recovery}

Two checks confirm that the elliptical results are properties of the
dynamics, not of the linear model.  First, the YA implementation
reduces to constant-coefficient HCW at $e = 0$: the $A$-matrix and
free-drift errors are at machine precision and the Simpson-rule
$B$-error is of order $10^{-9}$ (\cref{tab:ya_validation}).  Second,
and more importantly, the phase dependence survives nonlinear
propagation.  Replaying each YA extragradient control pair through a
full nonlinear two-body (and two-body-plus-$J_2$) integration of both
spacecraft, the terminal distances match the linear values to within
$0.011$~\si{\meter} for two-body truth and $1.1$~\si{\meter} with $J_2$,
and the factor-of-$1.7$ window at $e = 0.6$ is reproduced to three
significant figures.  Free-drift YA propagation likewise agrees with
independent Keplerian propagation to within $4.1$~\si{\milli\meter}
($0.001\%$ of separation) over $300$~\si{\second}.  The
orbital-phase effect is therefore real, not a linearization artifact.

\begin{table}[htbp]
\centering\footnotesize
\caption{YA-STM validation: relative errors at $e = 0$ against the
  HCW reference, plus structural checks.}
\label{tab:ya_validation}
\begin{tabular}{@{}lr@{}}
\toprule
Check & Relative error \\
\midrule
$\|A_{\mathrm{HCW}} - A_{\mathrm{YA}}(e\!=\!0)\| / \|A_{\mathrm{HCW}}\|$ & $7.96 \times 10^{-15}$ \\
$\|B_{\mathrm{HCW}} - B_{\mathrm{YA}}(e\!=\!0)\| / \|B_{\mathrm{HCW}}\|$ & $1.86 \times 10^{-9}$ \\
Free-drift trajectory (30 steps) & $2.22 \times 10^{-15}$ \\
STM composition $\|\Phi_{20} - \Phi_{21}\Phi_{10}\|$ & $1.05 \times 10^{-14}$ \\
\bottomrule
\end{tabular}
\end{table}

\section{Discussion}\label{sec:discussion}

\subsection{What the Framework Guarantees, and What It Does Not}

Exact are the two one-sided certificates of
\cref{sec:capture_cert,sec:security}: they assume nothing about the
opponent's rationality, carry over unchanged to the elliptical case,
and survive nonlinear propagation.  Not guaranteed is a pure open-loop
equilibrium, since the payoff is
convex in both players (\cref{prop:curvature}), so the extragradient
strategy pair is a certified-good pursuer strategy with a plausible
evader reply, not a Nash point.  Linear keep-out half-spaces and
$\ell_2$-bounded thrust stay inside the tractable envelope; minimum-time
intercept, hard $\Delta V$ caps, and integer impulsive control leave it
and would require HJI or successive convexification.

\subsection{Limitations}\label{sec:limitations}

Four assumptions bound the practical reach of the results, and the
measurements above let us be specific about each.  \emph{Navigation.}
Both players are assumed to know the relative state.  In flight,
cooperative GNSS-based navigation reaches decimetre and
millimetre-per-second
accuracy~\cite{DAmico2012PRISMA}, at which the outcome is unaffected
(\cref{sec:solver_char}); angles-only navigation of a non-cooperative
target is metre-level and roughly ten percent of range along the
boresight~\cite{Ardaens2018AnglesOnly,Gaias2018AVANTI}, which
measurably lowers the inspector's capture rate.  A minimax-filtering
or robust-reachable-set extension~\cite{LiZhang2026RobustPE} that
inflates the reachable sets by the navigation covariance is the
natural next step and preserves the convex structure.  \emph{Actuation.}
The continuous box-bounded model abstracts away minimum impulse bits
and finite slew; a $100$~\si{\kilo\gram}-class inspector with
$1$~\si{\newton} thrusters, as flown on
PRISMA~\cite{Bodin2009PRISMA}, has impulse quantization of a few
micrometres per second per pulse, well below the navigation noise, and
pulse-width modulation reproduces the continuous plan to good
fidelity~\cite{BernelliZazzera1992PWM,Vazquez2017PWMMPC}.
\emph{Perturbations.}  Over $300$--$600$~\si{\second} at sub-kilometre
separation the differential $J_2$, drag, and third-body accelerations
are orders of magnitude below the control
authority~\cite{Schweighart2002J2}, consistent with the sub-percent
linearization error of \cref{sec:nonlinear_validation}.  \emph{Capture
radius.}  The $50$~\si{\meter} radius is an approach gate rather than a
docking tolerance; standard RPO practice reserves keep-out spheres of
this scale around a resident space object~\cite{Fehse2003}, and the
$r_{\mathrm{cap}}$ sweep of \cref{sec:design_implications} shows how the
design rules tighten toward docking-class radii.  Finally, the zero-sum
structure does not model cooperative or general-sum interactions,
which lie outside the present scope.

\emph{Closed-loop play.}  The receding-horizon wrapper of
\cref{sec:tradespace} is an empirical extension.  We prove neither
recursive feasibility nor convergence of the closed-loop play to an
equilibrium, and we claim neither; what is available online is the
per-step reachable-set screen, which tells an operator before committing
a control whether the next open-loop subproblem is solvable.  The gap is
shared by the current model-predictive game literature, including the
model-predictive Stackelberg solution of Liu et
al.~\cite{Liu2025MPStackelberg} and the closed-loop elliptical pursuit
law of Jia et al.~\cite{Jia2025ClosedLoop}.  The most promising route to
a genuine guarantee is a passive-safety construction of the kind Elango
et al.~\cite{Elango2025PassiveSafe} develop for rendezvous, in which
every committed control leaves behind an abort trajectory that stays
safe, so feasibility is inherited from step to step by construction.
Establishing that for a two-player game, where the abort must hold
against an adversarial target rather than a passive one, is open.

\emph{Operational constraints.}  The only constraint on either player
here is the thrust box.  Real inspections also run under illumination,
communication, and sensing limits, and it is worth stating which of
these the present machinery absorbs and which it does not, because the
answer is not uniform.  Three classes cost nothing structurally.
Relative-velocity and range limits are norms of affine functions of the
control sequence, hence second-order cone
constraints~\cite{Malyuta2022}.  A line-of-sight or boresight
requirement, that the relative position stay within a half-angle
$\theta$ of a fixed direction $\bm{b}$, is the cone
$\norm{\pos_k}_2 \cos\theta \leq \bm{b}^\top \pos_k$, again second-order
cone representable, and the same form covers a Sun phase-angle
\emph{upper} bound for illumination.  Should the target tumble,
$\bm{b}$ becomes a known function of time and the cone rotates with it,
which leaves each step convex but ties the admissible set to the
attitude history; reachability under exactly that rotating constraint
is treated in~\cite{iskender2026reachability}.  Communication and eclipse windows
restrict which steps may thrust or must coast, which is a linear
equality on the affected components.  Each of these keeps every
subproblem convex and can be added to the solver as it stands.  Two
classes do not.  Exclusion constraints -- keep-out zones, solar or
sensor blinding cones, minimum-range holds -- are complements of convex
sets and therefore nonconvex; they need the successive-convexification
treatment already used for keep-out zones in \cref{sec:obstacle}, which
buys exact constraint satisfaction at convergence and gives up the
one-shot solve.  And constraints tied to the inspector's own attitude
couple translation to rotation, which the planar or six-state
translational model does not carry at all.

The certificates deserve a separate word, since they are the paper's
rigorous content.  Their $O(n_u)$ closed form is the box support
function of \cref{sec:support}.  Adding any of the
convex constraints above shrinks a player's admissible set, so its
support values must be recomputed as
$\support(\bm{d}) = \max\{\bm{d}^\top \pos : \pos \text{ admissible}\}$,
a small convex program instead of a closed form.  The exactness of
\cref{thm:escape} survives that substitution unchanged, because the
proof needs only that $\support$ be the true maximum over the admissible
set; what is lost is the $O(n_u)$ evaluation, and with it the ability to
screen thousands of geometries in the time the full game takes once.
Which way the bounds move is also predictable: shrinking the target's
admissible set raises $\phi_N$ and weakens its escape guarantee, while
shrinking the inspector's set lowers $\phi_N$ and raises the security
value, so each player's operational constraints tighten its opponent's
certificate and loosen its own.

The framework also addresses only the translational approach.  What
follows a successful approach (capture or servicing of the target) is
a separate problem, spanning cable-driven capture
mechanisms~\cite{Liu2024CableDriven} and learning-based control of
free-floating servicing manipulators~\cite{Adde2026PSORLSMC}.
Coordinating the approach with the inspector's own attitude, for which
constrained and finite-time pointing control is well
developed~\cite{Li2025TimeSynchronized,Wang2025AntiUnwinding}, is
likewise future work, as is fault-tolerant guidance that must react to
actuator failures mid-approach, a capability demonstrated for
terminal-phase vision-based landing in adjacent aerospace
domains~\cite{Khaneghaei2025FaultyUAV} but not covered by the present
open-loop certificates.

\section{Conclusions}\label{sec:conclusions}

Terminal-distance orbital pursuit--evasion is decided by the geometry
of the two players' reachable sets, and that geometry gives two exact
guarantees, an escape certificate and a pursuer security bound, that
hold for any admissible opponent on both circular and elliptical
reference orbits.  A projected extragradient method supplies a concrete
strategy pair in about $25$~\si{\milli\second}, certified in place by
its best-response gaps, and the escape certificate predicted every one
of $200$ perturbed outcomes without error.  The framework is honest
about its edges: it certifies one-sided guarantees rather than a saddle
point, its accuracy is set by the relative navigation available against
a non-cooperative target, and on eccentric orbits engagement timing
should be budgeted in orbital arc.  Three questions remain open.  Can
the receding-horizon and multi-pursuer variants be given closed-loop
guarantees?  Does inflating the reachable sets by a navigation
covariance yield a robust certificate?  And how far do the exact
guarantees extend once $J_2$ and drag are folded into the time-varying
linearization?

\section*{Reproducibility Statement}

All algorithms, scenarios, and data-generation scripts are available
as open-source Python code at
\url{https://github.com/iskender9961/Astrodynamics_Games}, with one
seeded entry point per experiment reported here and the complete
parameter specification of \cref{app:params}.  The implementation
depends only on NumPy, SciPy, CVXPY, OSQP, and Clarabel.

\section*{Acknowledgments}

The author acknowledges the support of Nanyang Technological University, Singapore.

\section*{Declaration of generative AI and AI-assisted technologies in the writing process}

During the preparation of this work the author used Anthropic Claude in
order to condense and copy-edit prose, to reorganize figure layout code,
and to cross-check the manuscript against the reviewers' comments.  No
generative tool was used to produce, analyze, or interpret the research
data: all numerical results, proofs, and conclusions are the author's
own and are reproducible from the code cited above.  After using this
tool the author reviewed and edited the content as needed and takes full
responsibility for the content of the publication.

\bibliographystyle{elsarticle-num}
\bibliography{refs}

\begin{thebibliography}{10}
\expandafter\ifx\csname url\endcsname\relax
  \def\url#1{\texttt{#1}}\fi
\expandafter\ifx\csname urlprefix\endcsname\relax\def\urlprefix{URL }\fi
\expandafter\ifx\csname href\endcsname\relax
  \def\href#1#2{#2} \def\path#1{#1}\fi

\bibitem{Burak2020}
O.~B. Iskender, Model predictive control for spacecraft rendezvous and docking
  with uncooperative targets, Ph.D. thesis, Nanyang Technological University,
  Singapore (2020).
\newblock \href {https://doi.org/10.32657/10356/144018}
  {\path{doi:10.32657/10356/144018}}.

\bibitem{Weeden2014}
B.~C. Weeden, V.~Samson, Global counterspace capabilities: An open source
  assessment, Secure World Foundation (2024).

\bibitem{Hobbs2020}
K.~L. Hobbs, M.~L. Perez, E.~A. Butcher, J.~M. Hinks, A survey of space domain
  awareness capabilities and research, Journal of the Astronautical Sciences
  67~(4) (2020) 1426--1457.

\bibitem{Long2026AutoML}
X.~Long, J.~Chen, L.~Yang, H.~Huang, An emergency scheduling method based on
  {AutoML} for space maneuver objective tracking, Expert Systems with
  Applications 298 (2026) 129759.
\newblock \href {https://doi.org/10.1016/j.eswa.2025.129759}
  {\path{doi:10.1016/j.eswa.2025.129759}}.

\bibitem{Cheng2025ast_review}
L.~Cheng, Z.~Wang, Y.~Song, F.~Jiang, Spacecraft intelligent orbital game
  technology: A review, Aerospace Science and Technology 161 (2025) 110158.
\newblock \href {https://doi.org/10.1016/j.ast.2025.110158}
  {\path{doi:10.1016/j.ast.2025.110158}}.

\bibitem{Isaacs1965}
R.~Isaacs, Differential Games: A Mathematical Theory with Applications to
  Warfare and Pursuit, Control and Optimization, John Wiley \& Sons, 1965.

\bibitem{Mitchell2005}
I.~M. Mitchell, A.~M. Bayen, C.~J. Tomlin, A time-dependent {Hamilton--Jacobi}
  formulation of reachable sets for continuous dynamic games, IEEE Transactions
  on Automatic Control 50~(7) (2005) 947--957.

\bibitem{Tomlin2000}
C.~J. Tomlin, J.~Lygeros, S.~S. Sastry, A game theoretic approach to controller
  design for hybrid systems, Proceedings of the IEEE 88~(7) (2000) 949--970.

\bibitem{Bansal2017}
S.~Bansal, M.~Chen, S.~Herbert, C.~J. Tomlin, {Hamilton--Jacobi} reachability:
  Some recent theoretical advances and applications in unmanned airspace
  management, Annual Review of Control, Robotics, and Autonomous Systems 4
  (2021) 253--279.

\bibitem{Li2023cja}
Q.~Li, Y.~Yuan, J.~Gao, S.~Zhang, Intelligent game-based pursuit-evasion
  strategy for spacecraft proximity operations, Chinese Journal of Aeronautics
  36~(10) (2023) 262--277.

\bibitem{Yang2025cja_rl_review}
K.~Yang, A.~Shen, N.~Xu, F.~Deng, M.~Lu, C.~Chen, A review of reinforcement
  learning approaches for pursuit-evasion games, Chinese Journal of Aeronautics
  38 (2025).
\newblock \href {https://doi.org/10.1016/j.cja.2025.103940}
  {\path{doi:10.1016/j.cja.2025.103940}}.

\bibitem{Clohessy1960}
W.~H. Clohessy, R.~S. Wiltshire, Terminal guidance system for satellite
  rendezvous, Journal of the Aerospace Sciences 27~(9) (1960) 653--658.

\bibitem{Yamanaka2002}
K.~Yamanaka, F.~Ankersen, New state transition matrix for relative motion on an
  arbitrary elliptical orbit, Journal of Guidance, Control, and Dynamics 25~(1)
  (2002) 60--66.

\bibitem{Althoff2021}
M.~Althoff, Set propagation techniques for reachability analysis, Annual Review
  of Control, Robotics, and Autonomous Systems 4 (2021) 369--395.

\bibitem{Girard2005}
A.~Girard, Reachability of uncertain linear systems using zonotopes, Hybrid
  Systems: Computation and Control (2005) 291--305.

\bibitem{Korpelevich1976}
G.~M. Korpelevich, The extragradient method for finding saddle points and other
  problems, Ekonomika i Matematicheskie Metody 12 (1976) 747--756.

\bibitem{Nemirovski2004}
A.~Nemirovski, Prox-method with rate of convergence {$O(1/t)$} for variational
  inequalities with {Lipschitz} continuous monotone operators and smooth
  convex-concave saddle point problems, SIAM Journal on Optimization 15~(1)
  (2004) 229--251.

\bibitem{Ho1965}
Y.~C. Ho, A.~E. Bryson, S.~Baron, Differential games and optimal
  pursuit-evasion strategies, IEEE Transactions on Automatic Control 10~(4)
  (1965) 385--389.

\bibitem{Guelman1990}
M.~Guelman, A.~Shinar, Optimal guidance law in the plane, in: AIAA Guidance,
  Navigation, and Control Conference, 1990, pp. 1--10.
\newblock \href {https://doi.org/10.2514/6.1990-3455}
  {\path{doi:10.2514/6.1990-3455}}.

\bibitem{Stupik2012}
J.~Stupik, M.~Pontani, B.~A. Conway, Optimal pursuit-evasion spacecraft
  maneuvers in low {Earth} orbit, Acta Astronautica 77 (2012) 91--104.

\bibitem{Pontani2009}
M.~Pontani, B.~A. Conway, Numerical solution of the three-dimensional orbital
  pursuit-evasion game, Journal of Guidance, Control, and Dynamics 32~(2)
  (2009) 474--487.
\newblock \href {https://doi.org/10.2514/1.37962} {\path{doi:10.2514/1.37962}}.

\bibitem{Jagat2017}
A.~Jagat, A.~J. Sinclair, Optimization of spacecraft pursuit-evasion game
  trajectories in the {Hill} reference frame, in: AIAA/AAS Astrodynamics
  Specialist Conference, 2017, pp. 1--15.
\newblock \href {https://doi.org/10.2514/6.2017-5232}
  {\path{doi:10.2514/6.2017-5232}}.

\bibitem{Woodford2023}
N.~Woodford, M.~Harris, P.~W. Bettinger, Game-theoretic analysis of orbital
  pursuit-evasion under impulsive thrust, Journal of Spacecraft and Rockets
  60~(3) (2023) 834--848.

\bibitem{Li2020}
Z.~Li, J.~Zhu, Orbital pursuit-evasion game with continuous low-thrust
  propulsion, Aerospace Science and Technology 97 (2020) 105670.

\bibitem{Ye2021}
D.~Ye, M.~Shi, Z.~Sun, Satellite proximate pursuit-evasion game with different
  thrust levels, Aerospace Science and Technology 111 (2021) 106568.

\bibitem{Liu2020ast_thrust}
Y.~Liu, Y.~Liu, J.~Liang, Satellite proximate pursuit--evasion game with
  different thrust configurations, Aerospace Science and Technology 96 (2020)
  105572.
\newblock \href {https://doi.org/10.1016/j.ast.2019.105572}
  {\path{doi:10.1016/j.ast.2019.105572}}.

\bibitem{Hafer2024}
W.~T. Hafer, K.~C. Howell, D.~C. Folta, Pursuit-evasion differential games in
  cislunar space, Journal of Guidance, Control, and Dynamics 47~(5) (2024)
  903--920.

\bibitem{Geng2023}
Y.~Geng, Z.~Chen, C.~Li, Reachable set based maneuver strategy for spacecraft
  proximity operations, Acta Astronautica 209 (2023) 68--82.

\bibitem{Cavalieri2023}
K.~A. Cavalieri, E.~S. Davis, L.~Singh, Game-theoretic control for space domain
  awareness and non-cooperative rendezvous, Journal of Spacecraft and Rockets
  60~(6) (2023) 1899--1913.

\bibitem{Ma2024ast_deltaV}
H.~Ma, G.~Zhang, Delta-v analysis for impulsive orbital pursuit--evasion based
  on reachable domain coverage, Aerospace Science and Technology 150 (2024)
  109243.
\newblock \href {https://doi.org/10.1016/j.ast.2024.109243}
  {\path{doi:10.1016/j.ast.2024.109243}}.

\bibitem{Huo2024ast_encirclement}
M.~Huo, J.~Liu, N.~Qi, S.~Yu, Research on proximity strategies for
  pursuit--evasion game with non-cooperative targets in space, Aerospace
  Science and Technology 155 (2024) 109677.
\newblock \href {https://doi.org/10.1016/j.ast.2024.109677}
  {\path{doi:10.1016/j.ast.2024.109677}}.

\bibitem{Pang2024PreciseGradient}
B.~Pang, C.~Wen, H.~Han, D.~Qiao, Solving pursuit/evasion game along elliptical
  orbit by providing precise gradient, Journal of Guidance, Control, and
  Dynamics 47~(4) (2024) 797--807.
\newblock \href {https://doi.org/10.2514/1.G007025}
  {\path{doi:10.2514/1.G007025}}.

\bibitem{Jia2025ClosedLoop}
Z.~Jia, D.~Ye, Y.~Xiao, Z.~Sun, Closed-loop strategy synthesis for real-time
  spacecraft pursuit--evasion games in elliptical orbits, IEEE Transactions on
  Aerospace and Electronic Systems 61~(4) (2025) 9071--9086.
\newblock \href {https://doi.org/10.1109/TAES.2025.3552066}
  {\path{doi:10.1109/TAES.2025.3552066}}.

\bibitem{Jia2025Reachability}
Z.~Jia, D.~Ye, Y.~Xiao, Z.~Sun, Approximate analytical approach for spacecraft
  pursuit--evasion game with reachability analysis, IEEE Transactions on
  Aerospace and Electronic Systems 61~(4) (2025) 9058--9070.
\newblock \href {https://doi.org/10.1109/TAES.2025.3552073}
  {\path{doi:10.1109/TAES.2025.3552073}}.

\bibitem{LiLuo2025Stackelberg}
Z.~Li, Y.~Luo, Orbital impulsive pursuit--evasion game formulation and
  {Stackelberg} equilibrium solutions, Journal of Spacecraft and Rockets 62~(2)
  (2025) 349--364.
\newblock \href {https://doi.org/10.2514/1.A35956}
  {\path{doi:10.2514/1.A35956}}.

\bibitem{Sun2025ThreePlayer}
S.~Sun, H.~Zhu, W.~Wang, Orbital three-player pursuit-evasion game, The Journal
  of the Astronautical Sciences 72~(3) (2025).
\newblock \href {https://doi.org/10.1007/s40295-025-00501-x}
  {\path{doi:10.1007/s40295-025-00501-x}}.

\bibitem{Zhao2025TwoOnOne}
L.~Zhao, Q.~Sun, Z.~Dang, Intelligent strategy resolution methods and mechanism
  analysis in two-on-one impulsive orbital pursuit--evasion games,
  Astrodynamics 9~(5) (2025) 727--751.
\newblock \href {https://doi.org/10.1007/s42064-025-0262-8}
  {\path{doi:10.1007/s42064-025-0262-8}}.

\bibitem{Fu2026TAD}
S.~Fu, S.~Gong, D.~Wu, P.~Shi, Analytical solutions and winning conditions of
  elliptic-orbit target-attacker-defender game, Journal of Guidance, Control,
  and Dynamics 49~(1) (2026) 257--268.
\newblock \href {https://doi.org/10.2514/1.G009187}
  {\path{doi:10.2514/1.G009187}}.

\bibitem{Yang2026WinningRegions}
Q.~Yang, Z.~Li, Y.-z. Luo, X.~Guo, Approximate analytical winning regions of
  the spacecraft in the coplanar pursuit--evasion--defense game, Journal of
  Spacecraft and Rockets 63~(3) (2026) 965--985.
\newblock \href {https://doi.org/10.2514/1.A36494}
  {\path{doi:10.2514/1.A36494}}.

\bibitem{Hu2024ast_multiagent}
C.~Hu, L.~Yang, Y.~Zhu, Impulsive maneuver strategy for multi-agent orbital
  pursuit--evasion game under sparse rewards, Aerospace Science and Technology
  155 (2024) 109599.
\newblock \href {https://doi.org/10.1016/j.ast.2024.109599}
  {\path{doi:10.1016/j.ast.2024.109599}}.

\bibitem{Huang2026ast_diverse}
S.~Huang, Diverse strategy generation for orbital pursuit--evasion games using
  distributed reinforcement learning, Aerospace Science and Technology 168
  (2026) 111052.
\newblock \href {https://doi.org/10.1016/j.ast.2025.111052}
  {\path{doi:10.1016/j.ast.2025.111052}}.

\bibitem{LiYe2022ast_switching}
Y.~Li, D.~Ye, Pursuit--evasion game switching strategies for spacecraft with
  incomplete information, Aerospace Science and Technology 119 (2022) 107200.
\newblock \href {https://doi.org/10.1016/j.ast.2021.107200}
  {\path{doi:10.1016/j.ast.2021.107200}}.

\bibitem{Wang2025LSTM}
H.~Wang, Y.~Zhang, S.~Bi, Game strategy prediction for spacecraft orbital
  pursuit--evasion game based on long short-term memory, Space: Science \&
  Technology 5 (2025) 0279.
\newblock \href {https://doi.org/10.34133/space.0279}
  {\path{doi:10.34133/space.0279}}.

\bibitem{Tschauner1965}
J.~Tschauner, P.~Hempel, Rendezvous zu einem in elliptischer {Bahn} umlaufenden
  {Ziel}, Astronautica Acta 11~(2) (1965) 104--109.

\bibitem{Carter1998}
T.~E. Carter, State transition matrices for terminal rendezvous studies: Brief
  survey and new example, Journal of Guidance, Control, and Dynamics 21~(1)
  (1998) 148--155.

\bibitem{Broucke2003}
R.~A. Broucke, Solution of the elliptic rendezvous problem with the time as
  independent variable, Journal of Guidance, Control, and Dynamics 26~(4)
  (2003) 615--621.

\bibitem{Inalhan2002}
G.~Inalhan, M.~Tillerson, J.~P. How, Relative dynamics and control of
  spacecraft formations in eccentric orbits, Journal of Guidance, Control, and
  Dynamics 25~(1) (2002) 48--59.

\bibitem{Gim2003}
D.-W. Gim, K.~T. Alfriend, State transition matrix of relative motion for the
  perturbed noncircular reference orbit, Journal of Guidance, Control, and
  Dynamics 26~(6) (2003) 956--971.

\bibitem{Vazquez2021ast_riccati}
R.~V{\'a}zquez, F.~Gavilan, E.~F. Camacho, Approximate analytic solution of
  nonlinear {Riccati} spacecraft formation flying dynamics in terms of orbit
  element differences, Aerospace Science and Technology 110 (2021) 106494.
\newblock \href {https://doi.org/10.1016/j.ast.2021.106494}
  {\path{doi:10.1016/j.ast.2021.106494}}.

\bibitem{Sullivan2017}
J.~Sullivan, S.~Grimberg, S.~{D'Amico}, Comprehensive survey and assessment of
  spacecraft relative motion dynamics models, Journal of Guidance, Control, and
  Dynamics 40~(4) (2017) 723--760.

\bibitem{Fu2025ast_keplerian}
S.~Fu, S.~Gong, P.~Shi, Analytical pursuit--evasion game strategy in arbitrary
  {Keplerian} reference orbits, Aerospace Science and Technology 158 (2025)
  109946.
\newblock \href {https://doi.org/10.1016/j.ast.2025.109946}
  {\path{doi:10.1016/j.ast.2025.109946}}.

\bibitem{Zhang2025ast_tfc}
C.~Zhang, Y.~Zhu, L.~Yang, A theory of functional connections-based method for
  orbital pursuit--evasion games with analytic satisfaction of rendezvous
  constraints, Aerospace Science and Technology 161 (2025) 110141.
\newblock \href {https://doi.org/10.1016/j.ast.2025.110141}
  {\path{doi:10.1016/j.ast.2025.110141}}.

\bibitem{Chen2018}
M.~Chen, C.~J. Tomlin, {Hamilton--Jacobi} reachability: Some recent theoretical
  advances and applications in unmanned airspace management, Annual Review of
  Control, Robotics, and Autonomous Systems 1 (2018) 333--358.

\bibitem{Fisac2015}
J.~F. Fisac, M.~Chen, C.~J. Tomlin, S.~S. Shankar, Reach-avoid problems with
  time-varying dynamics, targets and constraints, Hybrid Systems: Computation
  and Control (2015) 11--20.

\bibitem{Herbert2017}
S.~L. Herbert, M.~Chen, S.~Han, S.~Bansal, J.~F. Fisac, C.~J. Tomlin,
  {FaSTrack}: A modular framework for fast and guaranteed safe motion planning,
  in: IEEE Conference on Decision and Control, 2017, pp. 1517--1522.

\bibitem{Kurzhanski2000}
A.~B. Kurzhanski, P.~Varaiya, Ellipsoidal techniques for reachability analysis,
  in: Hybrid Systems: Computation and Control, 2000, pp. 202--214.

\bibitem{Shao2023cja_reachable}
L.~Shao, H.~Miao, R.~Hu, H.~Liu, Reachable set estimation for spacecraft
  relative motion based on bang-bang principle, Chinese Journal of Aeronautics
  36~(2) (2023) 229--240.
\newblock \href {https://doi.org/10.1016/j.cja.2022.07.003}
  {\path{doi:10.1016/j.cja.2022.07.003}}.

\bibitem{Zhang2025cja_reachable}
S.~Zhang, Z.~Yang, Y.~Luo, Spacecraft multi-impulse reachable domain, Chinese
  Journal of Aeronautics 38~(6) (2025).
\newblock \href {https://doi.org/10.1016/j.cja.2024.12.001}
  {\path{doi:10.1016/j.cja.2024.12.001}}.

\bibitem{Zhang2025TimeDependentRD}
S.~Zhang, Z.~Yang, Y.-Z. Luo, Time-dependent reachable domain and its
  application to impulsive orbital pursuit--evasion analysis, Journal of
  Spacecraft and Rockets 62~(2) (2025) 631--642.
\newblock \href {https://doi.org/10.2514/1.A36112}
  {\path{doi:10.2514/1.A36112}}.

\bibitem{Xu2026NashReachable}
W.~Xu, X.~Liu, Z.~Lu, B.~Hua, Y.~Wu, Numerical method for {Nash} equilibrium
  strategies of spacecraft orbit pursuit-evasion game based on continuous
  thrust reachable domain analysis, Science China Technological Sciences 69~(4)
  (2026).
\newblock \href {https://doi.org/10.1007/s11431-025-3186-8}
  {\path{doi:10.1007/s11431-025-3186-8}}.

\bibitem{iskender2026reachability}
O.~B. Iskender, K.~V. Ling, W.~S. Lim, E.~Lansard, Reachability-based
  safe-start regions for approach to a tumbling target with rotating {LOS}
  constraints, in: 77th International Astronautical Congress, IAC, Antalya,
  T\"{u}rkiye, 2026.

\bibitem{Rosen1965}
J.~B. Rosen, Existence and uniqueness of equilibrium points for concave
  {$N$}-person games, Econometrica 33~(3) (1965) 520--534.

\bibitem{vonNeumannMorgenstern1944}
J.~{von Neumann}, O.~Morgenstern, Theory of Games and Economic Behavior,
  Princeton University Press, 1944.

\bibitem{BoydVandenberghe2004}
S.~Boyd, L.~Vandenberghe, Convex Optimization, Cambridge University Press,
  2004.

\bibitem{OSQP2020}
B.~Stellato, G.~Banjac, P.~Goulart, A.~Bemporad, S.~Boyd, {OSQP}: An operator
  splitting solver for quadratic programs, Mathematical Programming Computation
  12~(4) (2020) 637--672.

\bibitem{Tseng1995}
P.~Tseng, On linear convergence of iterative methods for the variational
  inequality problem, Journal of Computational and Applied Mathematics 60~(1-2)
  (1995) 237--252.

\bibitem{Mokhtari2020}
A.~Mokhtari, A.~Ozdaglar, S.~Pattathil, A unified analysis of extra-gradient
  and optimistic gradient methods for saddle point problems: Proximal point
  approach, in: International Conference on Artificial Intelligence and
  Statistics (AISTATS), 2020, pp. 1497--1507.

\bibitem{Facchinei2003}
F.~Facchinei, J.-S. Pang, Finite-Dimensional Variational Inequalities and
  Complementarity Problems, Springer, 2003.

\bibitem{LeCleach2022ALGAMES}
S.~Le~Cleac'h, M.~Schwager, Z.~Manchester, {ALGAMES}: a fast augmented
  {Lagrangian} solver for constrained dynamic games, Autonomous Robots 46~(1)
  (2022) 201--215.
\newblock \href {https://doi.org/10.1007/s10514-021-10024-7}
  {\path{doi:10.1007/s10514-021-10024-7}}.

\bibitem{SoFan2023Epigraph}
O.~So, C.~Fan, Solving stabilize-avoid via epigraph form optimal control using
  deep reinforcement learning, in: Robotics: Science and Systems XIX, Daegu,
  Republic of Korea, 2023.
\newblock \href {https://doi.org/10.15607/RSS.2023.XIX.085}
  {\path{doi:10.15607/RSS.2023.XIX.085}}.

\bibitem{Malyuta2022}
D.~Malyuta, T.~P. Reynolds, M.~Szmuk, T.~Lew, R.~Bonalli, M.~Pavone,
  B.~A\c{c}{\i}kme\c{s}e, Convex optimization for trajectory generation: A
  tutorial, IEEE Control Systems Magazine 42~(5) (2022) 40--113.

\bibitem{AcikmeseaPloen2007}
B.~A\c{c}{\i}kme\c{s}e, S.~R. Ploen, Convex programming approach to powered
  descent guidance for {Mars} landing, Journal of Guidance, Control, and
  Dynamics 30~(5) (2007) 1353--1366.
\newblock \href {https://doi.org/10.2514/1.27553} {\path{doi:10.2514/1.27553}}.

\bibitem{Rawlings2017}
J.~B. Rawlings, D.~Q. Mayne, M.~M. Diehl, Model Predictive Control: Theory,
  Computation, and Design, 2nd Edition, Nob Hill Publishing, 2017.

\bibitem{Iskender2019}
O.~B. Iskender, K.-V. Ling, L.~Simonini, M.~Schlotterer, D.~Seelbinder,
  S.~Theil, J.~M. Maciejowski, Dual quaternion based autonomous rendezvous and
  docking via model predictive control, in: Proc. 70th Int. Astronautical
  Congr. Int. Astronautical Federation, 2019, pp. 1--16.

\bibitem{Mammarella2018ast_tubeMPC}
M.~Mammarella, E.~Capello, H.~Park, G.~Guglieri, M.~Romano, Tube-based robust
  model predictive control for spacecraft proximity operations in the presence
  of persistent disturbance, Aerospace Science and Technology 77 (2018)
  585--594.
\newblock \href {https://doi.org/10.1016/j.ast.2018.04.009}
  {\path{doi:10.1016/j.ast.2018.04.009}}.

\bibitem{Sanchez2020ast_NRHO}
J.~C. S{\'a}nchez, F.~Gavilan, R.~Vazquez, Chance-constrained model predictive
  control for {Near Rectilinear Halo Orbit} spacecraft rendezvous, Aerospace
  Science and Technology 100 (2020) 105827.
\newblock \href {https://doi.org/10.1016/j.ast.2020.105827}
  {\path{doi:10.1016/j.ast.2020.105827}}.

\bibitem{Oguri2024ChanceConstrained}
K.~Oguri, Chance-constrained control for safe spacecraft autonomy: Convex
  programming approach, in: Proceedings of the 2024 American Control Conference
  (ACC), 2024, pp. 2318--2324.
\newblock \href {https://doi.org/10.23919/ACC60939.2024.10645008}
  {\path{doi:10.23919/ACC60939.2024.10645008}}.

\bibitem{Hill1878}
G.~W. Hill, Researches in the lunar theory, American Journal of Mathematics
  1~(1) (1878) 5--26.

\bibitem{Richards2002}
A.~Richards, T.~Schouwenaars, J.~P. How, E.~Feron, Spacecraft trajectory
  planning with avoidance constraints using mixed-integer linear programming,
  Journal of Guidance, Control, and Dynamics 25~(4) (2002) 755--764.

\bibitem{Breeden2023RobustCBF}
J.~Breeden, D.~Panagou, Robust control barrier functions under high relative
  degree and input constraints for satellite trajectories, Automatica 155
  (2023) 111109.
\newblock \href {https://doi.org/10.1016/j.automatica.2023.111109}
  {\path{doi:10.1016/j.automatica.2023.111109}}.

\bibitem{vanWijk2024DCBF}
D.~van Wijk, K.~Dunlap, M.~Majji, K.~L. Hobbs, Safe spacecraft inspection via
  deep reinforcement learning and discrete control barrier functions, Journal
  of Aerospace Information Systems 21~(12) (2024) 996--1013.
\newblock \href {https://doi.org/10.2514/1.I011391}
  {\path{doi:10.2514/1.I011391}}.

\bibitem{CVXPY2016}
S.~Diamond, S.~Boyd, {CVXPY}: A {Python}-embedded modeling language for convex
  optimization, Journal of Machine Learning Research 17~(83) (2016) 1--5.

\bibitem{DAmico2012PRISMA}
S.~D'Amico, J.-S. Ardaens, R.~Larsson, Spaceborne autonomous formation-flying
  experiment on the {PRISMA} mission, Journal of Guidance, Control, and
  Dynamics 35~(3) (2012) 834--850.
\newblock \href {https://doi.org/10.2514/1.55638} {\path{doi:10.2514/1.55638}}.

\bibitem{Gaias2018AVANTI}
G.~Gaias, J.-S. Ardaens, Flight demonstration of autonomous noncooperative
  rendezvous in low {Earth} orbit, Journal of Guidance, Control, and Dynamics
  41~(6) (2018) 1337--1354.
\newblock \href {https://doi.org/10.2514/1.G003239}
  {\path{doi:10.2514/1.G003239}}.

\bibitem{Fehse2003}
W.~Fehse, Automated Rendezvous and Docking of Spacecraft, Cambridge University
  Press, Cambridge, UK, 2003.
\newblock \href {https://doi.org/10.1017/CBO9780511543388}
  {\path{doi:10.1017/CBO9780511543388}}.

\bibitem{Ardaens2018AnglesOnly}
J.-S. Ardaens, G.~Gaias, Flight demonstration of spaceborne real-time
  angles-only navigation to a noncooperative target in low {Earth} orbit, Acta
  Astronautica 153 (2018) 367--382.
\newblock \href {https://doi.org/10.1016/j.actaastro.2018.01.044}
  {\path{doi:10.1016/j.actaastro.2018.01.044}}.

\bibitem{LiZhang2026RobustPE}
Y.~Li, G.~Zhang, Robust trajectory optimization for pursuit-evasion game with
  navigation and control errors, Advances in Space Research 77~(2) (2026)
  2028--2042.
\newblock \href {https://doi.org/10.1016/j.asr.2025.10.058}
  {\path{doi:10.1016/j.asr.2025.10.058}}.

\bibitem{Bodin2009PRISMA}
P.~Bodin, R.~Larsson, F.~Nilsson, C.~Chasset, R.~Noteborn, M.~Nylund, {PRISMA}:
  An in-orbit test bed for guidance, navigation, and control experiments,
  Journal of Spacecraft and Rockets 46~(3) (2009) 615--623.
\newblock \href {https://doi.org/10.2514/1.40161} {\path{doi:10.2514/1.40161}}.

\bibitem{BernelliZazzera1992PWM}
F.~Bernelli-Zazzera, P.~Mantegazza, Pulse-width equivalent to pulse-amplitude
  discrete control of linear systems, Journal of Guidance, Control, and
  Dynamics 15~(2) (1992) 461--467.
\newblock \href {https://doi.org/10.2514/3.20858} {\path{doi:10.2514/3.20858}}.

\bibitem{Vazquez2017PWMMPC}
R.~Vazquez, F.~Gavilan, E.~F. Camacho, Pulse-width predictive control for {LTV}
  systems with application to spacecraft rendezvous, Control Engineering
  Practice 60 (2017) 199--210.
\newblock \href {https://doi.org/10.1016/j.conengprac.2016.06.017}
  {\path{doi:10.1016/j.conengprac.2016.06.017}}.

\bibitem{Schweighart2002J2}
S.~A. Schweighart, R.~J. Sedwick, High-fidelity linearized {J2} model for
  satellite formation flight, Journal of Guidance, Control, and Dynamics 25~(6)
  (2002) 1073--1080.
\newblock \href {https://doi.org/10.2514/2.4986} {\path{doi:10.2514/2.4986}}.

\bibitem{Liu2025MPStackelberg}
Y.~Liu, C.~Li, J.~Jiang, Y.~Zhang, A model predictive {Stackelberg} solution to
  orbital pursuit-evasion game, Chinese Journal of Aeronautics 38~(2) (2025)
  103198.
\newblock \href {https://doi.org/10.1016/j.cja.2024.08.029}
  {\path{doi:10.1016/j.cja.2024.08.029}}.

\bibitem{Elango2025PassiveSafe}
P.~Elango, A.~P. Vinod, K.~Kitamura, B.~A{\c{c}}{\i}kme{\c{s}}e, S.~Di~Cairano,
  A.~Weiss, Successive convexification for passively-safe spacecraft rendezvous
  on {Near Rectilinear Halo Orbit}, arXiv preprint arXiv:2505.17251, under
  review (2025).

\bibitem{Liu2024CableDriven}
R.~Liu, Y.~Fan, Y.~Huang, H.~Guo, C.~Zhao, M.~Su, Design and workspace analysis
  of a cable-driven space capture robot for noncooperative targets, Proceedings
  of the Institution of Mechanical Engineers, Part G: Journal of Aerospace
  Engineering 238~(14) (2024) 1406--1418.
\newblock \href {https://doi.org/10.1177/09544100241272826}
  {\path{doi:10.1177/09544100241272826}}.

\bibitem{Adde2026PSORLSMC}
Y.~A. Adde, Y.~N. Razoumny, A.~A. Betelie, T.~K. Mohammed, Y.~A. Wendemagegn,
  C.~M. Abdissa, Intelligent hybrid control for free-floating space robots:
  {PSO}--{RL}--{SMC} with inverse-dynamics tracking, IEEE Access 14 (2026)
  27137--27157.
\newblock \href {https://doi.org/10.1109/ACCESS.2026.3661861}
  {\path{doi:10.1109/ACCESS.2026.3661861}}.

\bibitem{Li2025TimeSynchronized}
D.~Li, S.~Tong, H.~Yang, Q.~Hu, Time-synchronized control for spacecraft
  reorientation with time-varying constraints, IEEE/ASME Transactions on
  Mechatronics 30~(3) (2025) 2073--2083.
\newblock \href {https://doi.org/10.1109/TMECH.2024.3430953}
  {\path{doi:10.1109/TMECH.2024.3430953}}.

\bibitem{Wang2025AntiUnwinding}
G.~Wang, Z.~Feng, Y.~Qu, H.~Sun, Event-triggered adaptive predefined-time
  anti-unwinding attitude tracking control for spacecraft, PLOS ONE 20~(10)
  (2025) e0333700.
\newblock \href {https://doi.org/10.1371/journal.pone.0333700}
  {\path{doi:10.1371/journal.pone.0333700}}.

\bibitem{Khaneghaei2025FaultyUAV}
M.~Khaneghaei, D.~Asadi, M.~Zahmatkesh, {\"O}.~Tutsoy, An experimental
  vision-based integrated guidance and control strategy for autonomous landing
  of a faulty {UAV}, Proceedings of the Institution of Mechanical Engineers,
  Part I: Journal of Systems and Control Engineering 239~(9) (2025) 1700--1716.
\newblock \href {https://doi.org/10.1177/09596518251339288}
  {\path{doi:10.1177/09596518251339288}}.

\end{thebibliography}

\clearpage
\appendix
\setcounter{table}{0}
\setcounter{figure}{0}
\renewcommand{\thetable}{\Alph{section}.\arabic{table}}
\renewcommand{\thefigure}{\Alph{section}.\arabic{figure}}
\makeatletter
\@addtoreset{table}{section}
\@addtoreset{figure}{section}
\makeatother

\section{Parameter Tables}\label{app:params}

\Cref{tab:params,tab:solver_params,tab:scenario_ics} specify the
complete set of parameters used in all numerical experiments.

\begin{table}[htbp]
\centering\footnotesize
\caption{Reference orbit, simulation and game parameters.  Values apply
  to every experiment unless a section states otherwise.}
\label{tab:params}
\begin{tabular}{@{}llrl@{}}
\toprule
Parameter & Symbol & Value & Unit / notes \\
\midrule
Semi-major axis          & $a$   & 6871    & km \\
Altitude (approx.)       & --    & 500     & km \\
Gravitational parameter  & $\mu$ & $3.986 \times 10^{14}$ & m$^3$/s$^2$ \\
Mean motion              & $n$   & $1.131 \times 10^{-3}$ & rad/s \\
Orbital period           & $T$   & 5556    & s \\
\midrule
Timestep                      & $\Delta t$         & 10.0   & s \\
Horizon (steps / time)        & $N$, $N \Delta t$  & 30, 300 & --, s \\
Position / state / control dim. & $n_p, n_x, n_u$  & 2, 4, 2 & -- \\
Inspector max acceleration    & $\bar{u}_P$        & 10.0   & mm/s$^2$ \\
Target max acceleration       & $\bar{u}_E$        & 5.0    & mm/s$^2$ \\
Inspector/target thrust ratio & $\bar{u}_P/\bar{u}_E$ & 2.0 & -- \\
Control norm type             & --                 & \multicolumn{2}{l}{Box ($\ell_\infty$)} \\
Capture radius                & $r_{\mathrm{cap}}$ & 50.0   & m \\
Effort regularization         & $\lambda$          & $10^{-3}$ & -- \\
Keep-out penalty weight       & $\kappa$           & $10^{3}$  & -- \\
Random seed                   & --                 & 42     & -- \\
\bottomrule
\end{tabular}
\end{table}

\begin{table}[htbp]
\centering\footnotesize
\caption{Solver, reachable-set and Monte Carlo parameters.  Monte Carlo
  capture is scored under the absorbing first-passage convention
  (\cref{app:mc}).}
\label{tab:solver_params}
\begin{tabular}{@{}llrl@{}}
\toprule
Parameter & Symbol & Value & Unit / notes \\
\midrule
\multicolumn{4}{@{}l}{\textit{IBR solver}} \\
Max iterations         & $K_{\mathrm{IBR}}^{\max}$ & 20   & Outer loop \\
Convergence tol.       & $\varepsilon_{\mathrm{IBR}}$ & $10^{-4}$ & On $|\Delta J|$ \\
OSQP iter.\ / abs.\ / rel.\ tol. & -- & \multicolumn{2}{l}{50\,000, $10^{-5}$, $10^{-5}$; also nominal MPC} \\
\midrule
\multicolumn{4}{@{}l}{\textit{Extragradient solver}} \\
Max iterations         & $K_{\mathrm{SP}}^{\max}$ & 200  & -- \\
Default step size      & $\eta_0$ & 0.01 & Before adaptation \\
Convergence tol.       & $\varepsilon_{\mathrm{SP}}$ & $10^{-4}$ & Relative change \\
Step size adapt.       & --   & \multicolumn{2}{l}{$\eta\!=\!\min(\eta_0, 0.5/(\|G\|^2\!+\!\lambda))$} \\
\midrule
\multicolumn{4}{@{}l}{\textit{Reachable sets and Monte Carlo}} \\
Template directions          & $L$  & 48   & 2D uniform, position space \\
Number of trials             & $N_{\mathrm{MC}}$ & 200  & Uniform perturbations \\
Inspector pos.\ / vel.\ range & $\Delta r, \Delta v$ & $\pm 50$, $\pm 0.5$ & m, m/s \\
Target pos.\ / vel.\ range & $\Delta r_E, \Delta v_E$ & $\pm 20$, $\pm 0.2$ & m, m/s \\
\bottomrule
\end{tabular}
\end{table}

\begin{table}[htbp]
\centering\footnotesize
\caption{Initial conditions and keep-out zones.  Positions in meters,
  velocities in m/s; the state vector is
  $\state_0 = [x, y, \dot{x}, \dot{y}]^\top$ and all agents start at
  rest.  The three keep-out balls apply to Case~B.}
\label{tab:scenario_ics}
\begin{tabular}{@{}llrrr@{}}
\toprule
Scenario & Agent & $x_0$ (m) & $y_0$ (m) & Radius (m) \\
\midrule
\multirow{2}{*}{Case A} & Inspector & 200  & $-300$ & -- \\
                        & Target    & 0    & 0      & -- \\
\midrule
\multirow{2}{*}{Case B} & Inspector & 300  & $-400$ & -- \\
                        & Target    & 0    & 0      & -- \\
\midrule
\multirow{3}{*}{Case C (2P)} & Inspector 1 & 300  & $-200$ & -- \\
                              & Inspector 2 & $-200$ & $-350$ & -- \\
                              & Target      & 0    & 0      & -- \\
\midrule
\multirow{4}{*}{Case C (3P)} & Inspector 1 & 300  & $-200$ & -- \\
                              & Inspector 2 & $-200$ & $-350$ & -- \\
                              & Inspector 3 & $-100$ & 300    & -- \\
                              & Target      & 0    & 0      & -- \\
\midrule
\multirow{3}{*}{Keep-out zones} & $\mathcal{O}_1$ & 100 & $-100$ & 60 \\
                              & $\mathcal{O}_2$ & $-50$ & $-200$ & 50 \\
                              & $\mathcal{O}_3$ & 150 & 100     & 55 \\
\bottomrule
\end{tabular}
\end{table}

The zero-order-hold discretization was checked against an RK45
continuous-integration reference ($\Delta t = 0.1$~s, 300~s horizon).
At the $\Delta t = 10$~s used throughout, the terminal position error
is $3.0 \times 10^{-2}$~m, a relative error below $0.01\%$ and well
inside the modeling error of the HCW linearization itself; the error
scales as $O(\Delta t^{2})$, reaching $7.4 \times 10^{-1}$~m
($0.24\%$) only at $\Delta t = 50$~s.

\section{Iterative Best Response Pseudocode}\label{app:algorithms}

\begin{algorithm}[!htbp]
\caption{Iterative Best Response for Two-Player PE Game}
\label{alg:ibr}
\small
\begin{algorithmic}[1]
\REQUIRE Initial states $\state_0^{(P)}, \state_0^{(E)}$; dynamics $A_d, B_d$;
  bounds $\bar{u}_P, \bar{u}_E$; horizon $N$; tolerance $\varepsilon$;
  max iterations $K_{\max}$
\ENSURE Control sequences $\bm{U}^{(P)*}, \bm{U}^{(E)*}$; trajectories
  $\bm{X}^{(P)*}, \bm{X}^{(E)*}$
\STATE Initialize $\bm{U}^{(E)} \leftarrow \bm{0}$
\STATE Propagate $\bm{X}^{(E)} \leftarrow \text{Propagate}(\state_0^{(E)}, \bm{U}^{(E)}, A_d, B_d)$
\STATE $J_{\mathrm{prev}} \leftarrow \infty$
\FOR{$t = 1, 2, \ldots, K_{\max}$}
  \STATE \COMMENT{Pursuer best response}
  \STATE $\bm{U}^{(P)} \leftarrow \argmin_{\|\ctrl\|_\infty \leq \bar{u}_P}
    \|\pos_N^{(P)} - \pos_N^{(E)}\|^2 + \lambda \|\bm{U}^{(P)}\|^2$
    \hfill (QP via OSQP)
  \STATE $\bm{X}^{(P)} \leftarrow \text{Propagate}(\state_0^{(P)}, \bm{U}^{(P)}, A_d, B_d)$
  \STATE \COMMENT{Evader best response}
  \STATE Compute escape direction $\bm{d} \leftarrow (\pos_N^{(E)} - \pos_N^{(P)})/\norm{\pos_N^{(E)} - \pos_N^{(P)}}$
    \hfill ($\pos_N^{(E)}$: current-iterate terminal position)
  \IF{$\norm{\pos_N^{(E)} - \pos_N^{(P)}} < \varepsilon_{\bm{d}}$}
    \STATE $\bm{d} \leftarrow [0,\, 1]^\top$
      \hfill (along-track fallback, $\varepsilon_{\bm{d}} = 10^{-3}$~m)
  \ENDIF
  \STATE $\bm{U}^{(E)} \leftarrow \argmax_{\|\ctrl\|_\infty \leq \bar{u}_E}
    \bm{d}^\top \pos_N^{(E)} - \lambda \|\bm{U}^{(E)}\|^2$
    \hfill (QP via OSQP)
  \STATE $\bm{X}^{(E)} \leftarrow \text{Propagate}(\state_0^{(E)}, \bm{U}^{(E)}, A_d, B_d)$
  \STATE \COMMENT{Convergence check}
  \STATE $J \leftarrow \|\pos_N^{(P)} - \pos_N^{(E)}\|$
  \IF{$|J - J_{\mathrm{prev}}| < \varepsilon$}
    \RETURN $\bm{U}^{(P)}, \bm{U}^{(E)}, \bm{X}^{(P)}, \bm{X}^{(E)}$, ``converged''
  \ENDIF
  \STATE $J_{\mathrm{prev}} \leftarrow J$
\ENDFOR
\RETURN $\bm{U}^{(P)}, \bm{U}^{(E)}, \bm{X}^{(P)}, \bm{X}^{(E)}$, ``max\_iter''
\end{algorithmic}
\end{algorithm}

\noindent
The escape direction is taken from the evader's \emph{current}
trajectory iterate: $\pos_N^{(E)}$ is the terminal position of the
propagated evader trajectory at the start of the iteration, which on
the first pass is the free-drift terminal position because
$\bm{U}^{(E)}$ is initialized to zero.  This matches the
implementation.  When the
two terminal positions nearly coincide, the direction is ill-defined
and the solver defaults to the along-track axis $[0,1]^\top$.  The
pursuer best response and the convergence check are unchanged.

The remaining two procedures need no pseudocode, being direct
transcriptions of equations in the main text.  The extragradient solver
precomputes $G_P, G_E, \bm{g}_0$ from \cref{eq:traj_affine,eq:delta}
and the step size $\eta = \min(\eta_0, 0.5/(\norm{G}_2^2 + \lambda))$,
starts from $\bm{U}^{(P)} = \bm{U}^{(E)} = \bm{0}$, and then alternates
\cref{eq:eg_extrap,eq:eg_update} with the gradients
\cref{eq:grad_p,eq:grad_e}, the projection $\Pi_{\mathcal{Z}}$ being
element-wise clipping to $[-\bar{u}_j, \bar{u}_j]$; it terminates when
the relative change in $J$ \cref{eq:payoff} falls below
$\varepsilon_{\mathrm{SP}}$ or at $K_{\mathrm{SP}}^{\max}$ iterations
(\cref{tab:solver_params}).  Reachable-set propagation evaluates
\cref{eq:sf_reach} along each template direction, accumulating the box
support value $\bar{u}\norm{\cdot}_1$ over the horizon, and in the
planar case intersects consecutive half-planes \cref{eq:polytope} to
recover the polytope vertices.

\section{Yamanaka--Ankersen State Transition Matrix}\label{app:ya_stm}

This appendix records only what is needed to reproduce our
implementation of the Yamanaka--Ankersen (YA) state transition matrix
of \cref{sec:elliptical}.  The underlying Tschauner--Hempel equations
for relative motion about an elliptical chief, written with true
anomaly $\nu$ as the independent variable and scaled by
$\rho(\nu) = 1 + e\cos\nu$, are standard~\cite{Tschauner1965}, as is
the YA derivation itself~\cite{Yamanaka2002}; neither is reproduced
here.  What matters for the LTV construction is the auxiliary
integral
\begin{equation}\label{eq:j_int_full}
\mathcal{J}(\nu,\nu_0) = \frac{h}{p^2}(t - t_0)
  = \int_{\nu_0}^{\nu} \frac{d\nu'}{\rho(\nu')^2}
  = \frac{M(\nu) - M(\nu_0)}{(1-e^2)^{3/2}},
\end{equation}
where $h = \sqrt{\mu p}$, $p = a(1-e^2)$, and
$M = \mathcal{E} - e\sin \mathcal{E}$ is
the mean anomaly obtained from $\nu$ via the eccentric anomaly
$\mathcal{E}$ (written thus so that $E$ remains the evader label).
The integral $\mathcal{J}$ removes the singular near-circular behavior of
earlier closed-form
solutions~\cite{Carter1998,Broucke2003} and is evaluated efficiently
through Kepler's equation.

The closed-form STM factors as $\Phi(\nu,\nu_0) =
\tilde{\Phi}(\nu)\,\tilde{\Phi}^{-1}(\nu_0)$, with $\tilde{\Phi}$ the
fundamental solution matrix of the homogeneous T-H system; its entries
are polynomials in $\rho(\nu)$, $\sin\nu$, $\cos\nu$ and
$\mathcal{J}$, and are tabulated in~\cite{Yamanaka2002}.  After
permuting from the T-H ordering to the HCW convention and correcting
the along-track sign, $\Phi(\nu,\nu_0)$ enters the discrete LTV model
\cref{eq:ltv_discrete} as $A_k = \Phi(\nu_{k+1},\nu_k)$, while
\begin{equation}\label{eq:bk_simpson}
B_k \approx \frac{\Delta t}{6}\left(
  \Phi(\nu_{k+1},\nu_k)\,\hat{B}
  + 4\,\Phi(\nu_{k+1},\nu_m)\,\hat{B}
  + \hat{B}\right),
\end{equation}
with $\hat{B} = [0;\; I_{n_u}]$ and $\nu_m$ the midpoint anomaly.  This
quadrature is $O(\Delta t^5)$ accurate; at $\Delta t = 10$~s the $B_k$
entries match the reference to $1.86 \times 10^{-9}$, and the
HCW-recovery and composition checks are in
\cref{sec:hcw_recovery,tab:ya_validation}.

The discrete matrices $A_k, B_k$ are time-varying because the
underlying flow is anomaly-driven.  Four distinct mechanisms produce
the variation, and all four collapse at $e = 0$, where $\rho \equiv 1$,
$\nu = n(t - t_0)$ and $\mathcal{J} = \nu - \nu_0$, returning the
constant HCW matrices $A_d, B_d$:
\begin{enumerate}
\item $\rho(\nu) = 1 + e\cos\nu$ enters both the fundamental solutions
  $\tilde{\Phi}$ and the coordinate scaling $T(\nu)$ that converts
  between physical and T-H variables, so $A_k = \Phi(\nu_{k+1},\nu_k)$
  depends on the true anomaly at \emph{both} endpoints of the step, not
  on the elapsed time alone.
\item The clock-to-anomaly map $\nu(t)$, obtained by inverting Kepler's
  equation, is nonlinear: a uniform time grid maps to a non-uniform
  anomaly grid that advances fastest near periapsis, so equal $\Delta t$
  steps span unequal $\Delta\nu$.
\item The propagation carries a secular term proportional to
  $\mathcal{J}(\nu_{k+1},\nu_k) \propto \Delta t$, which accumulates
  linearly with elapsed time.
\item The control matrix $B_k$ from the Simpson quadrature is modulated
  by $\rho$ along the arc, since $\hat{B}$ is mapped through
  $\Phi(\cdot,\nu)$ evaluated at anomaly-dependent nodes.
\end{enumerate}

\section{Regularization Sensitivity}\label{app:ablation}

\Cref{tab:ablation_lambda} gives the full sweep behind the
$\lambda$-invariance reported in \cref{sec:saddle_form}.  The plateau is
explained by saturation: while both thrust sequences stay saturated the
regularizer only shifts the objective by a constant, and the table's
saturation column shows where that ends.

\begin{table}[htbp]
\centering\footnotesize
\caption{Regularization sweep (Case A).  Terminal miss distance $d_f$,
  per-player $\Delta V$, thrust-saturation fraction, and extragradient
  iteration count versus the effort weight $\lambda$.}
\label{tab:ablation_lambda}
\begin{tabular}{@{}rrrrrrr@{}}
\toprule
$\lambda$ & $d_f$ (m) & $\Delta V_P$ (m/s) & $\Delta V_E$ (m/s)
  & sat$_P$ (\%) & sat$_E$ (\%) & Iters \\
\midrule
$10^{-4}$ to $10^{4}$\,$^\ddagger$ & 72.63 & 4.16 & 2.11 & 96.7 & 96.7 & 141 \\
$10^{5}$         & 72.71  & 4.15 & 2.10 & 95.0 & 96.7 & 131 \\
$10^{6}$         & 73.89  & 4.04 & 2.07 & 91.7 & 95.0 &  96 \\
$3\times10^{6}$  & 80.19  & 3.80 & 2.01 & 78.3 & 88.3 &  51 \\
$10^{7}$         & 109.59 & 3.28 & 1.89 & 53.3 & 78.3 &  26 \\
\bottomrule
\end{tabular}
\\[2pt]{\footnotesize $^\ddagger$ Eight sampled values ($10^{-4}, 10^{-3},
  10^{-2}, 10^{-1}, 10^{0}, 10^{1}, 10^{3}, 10^{4}$), all identical to the
  digits shown; $142$ iterations at $10^{4}$.}
\end{table}

\section{Monte Carlo Robustness Study}\label{app:mc}

A $200$-trial Monte Carlo study perturbs the Case~A initial state
(inspector $\pm 50$~m, $\pm 0.5$~m/s; target $\pm 20$~m,
$\pm 0.2$~m/s; uniform) and re-solves the engagement under each method.
Metrics use the absorbing first-passage convention, so the reported
$d_f$ is the separation at engagement end, at most
$r_{\mathrm{cap}} = 50$~m on captured trials.  Under game-theoretic play
the capture rate drops from $99.0\%$ (passive target, nominal MPC) to
$25.0\%$ (\cref{tab:mc}).  The extragradient and IBR solvers capture on
the same $50$ trials and produce statistically indistinguishable
terminal distances; a Welch test on $d_f$ gives $t = 0.123$,
$\mathrm{dof} = 398$, $p = 0.90$.  The extragradient solver runs about
$73\times$ faster on this machine.

\begin{table}[htbp]
\centering\footnotesize
\caption{Monte Carlo results ($200$ trials, Case A; first-passage
  convention).  Mean $\pm$ standard deviation; capture rate with a
  Wilson $95\%$ confidence interval.}
\label{tab:mc}
\begin{tabular}{@{}lrrcr@{}}
\toprule
Method & $d_f$ (m) & Cap.\ (\%) & $95\%$ CI & Time (s) \\
\midrule
Nominal MPC  & $38.3 \pm 8.5$   & 99.0 & $[96.4, 99.7]$ & $0.09$ \\
IBR          & $129.7 \pm 75.0$ & 25.0 & $[19.5, 31.4]$ & $2.34$ \\
Extragradient & $130.6 \pm 74.6$ & 25.0 & $[19.5, 31.4]$ & $0.03$ \\
\bottomrule
\end{tabular}
\end{table}

\Cref{tab:mc_by_d0} bins the trials by the initial inspector range
$d_0$.  The game-solver capture rate falls monotonically with $d_0$,
from $44.4\%$ in $[280, 320)$~m to $5.0\%$ in $[400, 440)$~m, so the
outcome is set primarily by the opening geometry.

\begin{table}[htbp]
\centering\footnotesize
\caption{Monte Carlo detail ($200$ trials, Case A).  Top: extragradient
  capture rate by initial inspector range $d_0$, with Wilson $95\%$
  intervals.  Bottom: escape-certificate confusion matrix against the
  extragradient outcome.  The $\phi_N \geq 0$ row scores $\phi_N$ used
  as a two-sided screen, which is more than \cref{thm:escape} claims;
  the theorem itself is silent there.}
\label{tab:mc_by_d0}
\begin{tabular}{@{}lrrc@{}}
\toprule
$d_0$ bin (m) & $n$ & Capture (\%) & $95\%$ CI \\
\midrule
$[280, 320)$ & 18 & 44.4 & $[25, 66]$ \\
$[320, 360)$ & 77 & 27.3 & $[19, 38]$ \\
$[360, 400)$ & 85 & 23.5 & $[16, 34]$ \\
$[400, 440)$ & 20 &  5.0 & $[1, 24]$  \\
\midrule
Certificate & EG captured & EG escaped & \\
\midrule
$\phi_N \geq 0$ (capture predicted) & 50 & 0   & \\
$\phi_N < 0$ (escape certified)     & 0  & 150 & \\
\bottomrule
\end{tabular}
\end{table}

The escape certificate $\phi_N$ predicts the extragradient outcome
without error across all $200$ trials (\cref{tab:mc_by_d0}).  Every trial
with $\phi_N \geq 0$ is captured ($50/50$; false-positive rate $0\%$,
one-sided $95\%$ upper bound $5.8\%$), and every trial with
$\phi_N < 0$ escapes ($0/150$; false-negative rate $0\%$, upper bound
$2.0\%$).  Over all trials $\mathrm{corr}(\phi_N, d_f) = -0.96$; on the
escape branch the fit $d_f = -0.997\,\phi_N + 51.1$ ($r = -0.9999$)
shows the certificate bound $r_{\mathrm{cap}} - \phi_N$ is tight.

\section{Algorithmic Complexity and Scalability}\label{app:scalability}

All three methods share the same $O(N^2 n_x n_u)$ precomputation of the
trajectory matrices.  They differ in the inner loop: nominal MPC and
IBR solve a quadratic program, at $O(N^3 n_u^3)$ per solve via
interior-point methods and hence $O(K_{\mathrm{IBR}} N^3 n_u^3)$ over
$K_{\mathrm{IBR}}$ best-response rounds, whereas the extragradient
iteration costs only $O(K_{\mathrm{SP}}\, n_P N n_u)$, \emph{linear} in
$N n_u$, because it performs matrix--vector products and projections
and nothing else.  Empirically, the extragradient loop costs
$89.4~\mu$s per iteration for Case~A ($12.6$~ms over $141$ iterations,
after an $11.0$~ms one-time matrix setup).

The absolute cost, however, depends heavily on how the QP inside IBR
is implemented (\cref{tab:timing}).  One-shot, the extragradient solver
is $41\times$ faster than as-is IBR that rebuilds the CVXPY problem each
call ($971.5$ vs $23.5$~ms) and $24.8\times$ faster than a
parametrized-CVXPY IBR.  Amortized, a condensed warm-started
direct-OSQP IBR solves the same game in $1.36$~ms after an $18.0$~ms
one-time setup; all three IBR variants agree at $d_f = 72.28$~m.  The
escape screen and the security bound bracket the game cost: the $\phi_N$
screen costs $10.8$~ms from scratch and $0.064$~ms once the terminal
maps are in hand, while the pursuer-security SOCP costs $177.4$~ms.

\begin{table}[htbp]
\centering\small
\caption{One-shot versus amortized solver cost for Case~A
  ($N = 30$, $n_u = 2$).  Median over $20$ repeats on an idle machine.}
\label{tab:timing}
\begin{tabular}{@{}lrrl@{}}
\toprule
Method & Median (ms) & Iters & Quality \\
\midrule
IBR, as-is (CVXPY rebuilt per call)    & 971.5 & 3   & $d_f = 72.28$~m \\
IBR, parametrized CVXPY (DPP)          & 583.9 & 3   & $d_f = 72.28$~m \\
IBR, condensed OSQP (warm-started)     & 1.36  & 4   & $d_f = 72.28$~m \\
\quad condensed-OSQP setup (once)      & 18.0  & --  & -- \\
Extragradient (full call)              & 23.5  & 141 & $d_f = 72.63$~m \\
Nominal MPC (single QP)                & 36.7  & --  & passive target \\
$\phi_N$ screen ($L=48$), from scratch & 10.8  & --  & $\phi_N = -22.2$~m \\
$\phi_N$ screen ($L=48$), maps reused  & 0.064 & --  & $\phi_N = -22.2$~m \\
Pursuer-security SOCP ($L=96$)         & 177.4 & --  & $\bar{V} = 319.7$~m \\
\bottomrule
\end{tabular}
\end{table}

For $n_P$ pursuers, the extragradient method scales as
$O(n_P \cdot n_p \cdot N n_u)$, linearly in the number of pursuers.
\Cref{tab:scaling} presents empirical timing data.

\begin{table}[htbp]
\centering\small
\caption{Empirical solve times (milliseconds) versus the number of
  pursuers $n_P$ ($N = 30$, $n_u = 2$; median over $20$ repeats).  The
  condensed IBR-multi column is warm-started (one OSQP object per
  agent, set up once).}
\label{tab:scaling}
\begin{tabular}{@{}rrrrr@{}}
\toprule
$n_P$ & MPC-multi & IBR-multi (as-is) & IBR-multi (cond.) & EG-multi \\
\midrule
1 & 36.5  & 961.5  & 1.45 & 28.8 \\
2 & 74.0  & 1102.7 & 1.69 & 25.5 \\
3 & 116.5 & 1218.5 & 3.45 & 36.0 \\
5 & 211.5 & 2619.7 & 10.6 & 75.8 \\
\bottomrule
\end{tabular}
\end{table}

The extragradient method scales roughly linearly in $n_P$ and retains a
$33$--$43\times$ speed advantage over as-is IBR across all agent counts.
A condensed, warm-started IBR is faster still once its per-agent setup is
amortized, so the speed ratio is not the operative advantage; what is,
is set out under \emph{Onboard implementation} below.

\paragraph{Higher-dimensional models}
The extragradient core is dimension-agnostic.  Each iteration is two
gradient evaluations and two box projections, all matrix--vector
products, so the cost grows linearly in $Nn_u$ rather than
combinatorially, in contrast to the exponential grid growth that
confines Hamilton--Jacobi--Isaacs reachability to the planar case.
Restoring out-of-plane translation is the case measured in
\cref{sec:scope_checks}, where six states cost $19$--$22$~ms against
$25$~ms planar.  Attitude
states enter through the control-effect matrix and are the natural next
extension; there the added difficulty is modeling fidelity, since
attitude couples nonlinearly, rather than solver complexity.

\paragraph{Onboard implementation}
The same structure is what recommends the method for flight hardware:
the iteration performs no matrix factorization and calls no compiled
optimization backend, so it needs neither a linear-algebra library with
dynamic memory nor a per-step problem build, and its working set is the
$O(Nn_u)$ iterate plus the precomputed terminal maps.  We measured
$89.4~\mu$s per iteration and about $25$~ms per Case~A solve on a
desktop and make no claim about a specific flight processor, which we
did not benchmark; the argument for onboard suitability is structural,
resting on the absence of factorization and compilation, not on a
ported measurement.

\section{Notation Summary}\label{app:notation}

Every symbol below carries one meaning, with two exceptions that are
distinguished by their argument.  $J$ is the scalar game payoff,
whereas $\mathcal{J}(\nu,\nu_0)$ is the Yamanaka--Ankersen
integral of \cref{eq:j_int_full}; and $h$ prints both for the
support-function operator $\support_{\mathcal{S}}(\cdot)$ and, in the
elliptical model, for the specific angular momentum $\sqrt{\mu p}$.
Three further collisions were removed rather than tolerated: the
eccentric anomaly is written $\mathcal{E}$ so that $E$ stays the evader
label, the geometric contraction factor of \cref{thm:convergence} is
written $q$, and the reachable-set excess-area ratio is written
$\alpha_L$, so that $\rho(\nu)$ denotes only the radius ratio
$1 + e\cos\nu$.  Altitude is given in prose (\SI{500}{\kilo\metre})
rather than as a symbol, since $h$ is taken.

\begin{table}[htbp]
\centering\footnotesize
\caption{Summary of notation used throughout the paper.  Dimensions are
  given where they are not implied by the definition; $n_x = 4$,
  $n_u = n_p = 2$ in all experiments.}
\label{tab:notation}
\begin{tabular}{@{}ll@{\hspace{1.2em}}ll@{}}
\toprule
Symbol & Description & Symbol & Description \\
\midrule
$\state$ & State vector, $\R^{n_x}$                     & $\lambda$ & Effort regularization weight \\
$\ctrl$ & Control (acceleration), $\R^{n_u}$            & $\kappa$ & Keep-out penalty weight \\
$\pos$ & Position vector, $\R^{n_p}$                    & $\eta$ & Extragradient step size \\
$A_c, B_c$ & Continuous-time matrices                   & $\norm{G}_2$ & $\max(\norm{G_P}_2, \norm{G_E}_2)$ \\
$A_d, B_d$ & Discrete-time matrices                     & $\sigma_{\max}(G_E)$ & Largest singular value of $G_E$ \\
$n$ & Mean motion, rad/s                                & $L_F$ & Lipschitz constant of $F$ \\
$\Delta t$ & Timestep, s                                & $\Pi_{\mathcal{Z}}$ & Projection onto $\mathcal{Z}$ \\
$N$ & Horizon length                                    & $d_f$ & Terminal miss distance, m \\
$\bar{u}_P, \bar{u}_E$ & Max acceleration, mm/s$^2$     & $\phi_N,\ \phi_N^{(i)}$ & Escape certificate, m \\
$r_{\mathrm{cap}}$ & Capture radius, m                  & $\bar{V}$ & Pursuer security value, m \\
$\ctrlset$ & Control constraint set                     & $\gamma_P$ & Pursuer best-response gap, m \\
$\reachset_N$ & $N$-step reachable set                  & $\psi_N$ & Capture-pair existence value, m \\
$\support_{\mathcal{S}}$ & Support function of $\mathcal{S}$ & $w_i$ & Adaptive pursuer weight \\
$\bm{d}_\ell$ & Template direction, $\R^{n_p}$          & $n_P$ & Number of inspectors \\
$L$ & Number of template directions                     & $\obstset_m$ & Keep-out zone $m$ \\
$\alpha_L$ & Reachable-set excess-area ratio            & $q$ & Contraction factor, \cref{thm:convergence} \\
 & & \multicolumn{2}{@{}l}{\textit{Elliptical orbit extension}} \\
 & & $e$ & Orbital eccentricity \\
$T_N$ & Terminal-position extractor                     & $\nu$ & True anomaly, rad \\
$S$ & Control-to-state matrix                           & $a$ & Semi-major axis, m \\
$G_P, G_E$ & Terminal position response                 & $p$ & Semi-latus rectum $a(1-e^2)$, m \\
$\bm{g}_0$ & Free-response terminal offset              & $h$ & Angular momentum $\sqrt{\mu p}$ \\
$\bm{\delta}$ & $G_P\bm{U}^{(P)}\!-\!G_E\bm{U}^{(E)}\!+\!\bm{g}_0$ & $\rho(\nu)$ & $1 + e\cos\nu$ \\
$J$ & Game payoff function                              & $\mathcal{J}(\nu,\nu_0)$ & YA fundamental integral \\
 & & $\Phi(\nu,\nu_0)$ & Yamanaka--Ankersen STM \\
 & & $A_k, B_k$ & Time-varying matrices at step $k$ \\
 & & $\Phi_{N:j}$ & Transition from step $j$ to $N$ \\
\bottomrule
\end{tabular}
\end{table}

\end{document}